\documentclass[11pt]{article}

\usepackage[utf8]{inputenc}
\usepackage[english]{babel}
\usepackage[T1]{fontenc}

\title{The heat flow on glued spaces with varying dimension}

\author{
  Anton Ullrich\thanks{Carnegie Mellon University, 5000 Forbes Avenue, 
Pittsburgh, PA 15213, United States. \nolinkurl{aullrich@andrew.cmu.edu}}
}

\date{}

\usepackage{amsthm}
\usepackage{stmaryrd}
\usepackage{enumitem}

\usepackage{algorithm2e}
\SetKwComment{Comment}{/* }{ */}
\RestyleAlgo{ruled}
\usepackage{listings}
\usepackage{verbatim}
\usepackage{stmaryrd}
\usepackage{dsfont}
\usepackage{tikz}
\usetikzlibrary{calc, quotes, angles}
\usepackage{float}
\usepackage{amsmath}
\usepackage{amssymb}
\usepackage{amsthm}
\usepackage{mathrsfs}

\usepackage{mathtools}
\usepackage{esint}
\usetikzlibrary{decorations.pathreplacing}

\usepackage{algorithm2e}

\PassOptionsToPackage{dvipsnames}{xcolor}
\RequirePackage{xcolor}
\definecolor{webgreen}{rgb}{0,.5,0}
\definecolor{webbrown}{rgb}{.6,0,0}
\definecolor{RoyalBlue}{cmyk}{1, 0.50, 0, 0}

\usepackage[colorlinks,
pdfpagelabels,
pdfstartview = FitH,
bookmarksopen = true,
bookmarksnumbered = true,
linkcolor = RoyalBlue,
plainpages = false,
urlcolor=webbrown,
hypertexnames = false,
citecolor = webgreen] {hyperref}

\newcommand*{\dd}{\mathop{}\!\mathrm{d}}

\newcommand{\norm}[1]{\left\lVert#1\right\rVert}
\newcommand{\hookdoubleheadrightarrow}{
  \hookrightarrow\mathrel{\mspace{-15mu}}\rightarrow
}
\newcommand{\R}{\mathbb{R}}
\newcommand{\Capp}{\text{Cap}}
\newcommand*{\LL}{\mathcal{L}}
\newcommand*{\vol}{\text{vol}}

\DeclareMathOperator*{\essinf}{ess\,inf}
\DeclareMathOperator*{\supp}{supp}

\DeclareMathOperator*{\Lip}{Lip}

\DeclareMathOperator*{\spn}{span}
\DeclareMathOperator*{\Id}{Id}

\DeclareMathOperator*{\E}{\mathcal{E}}
\DeclareMathOperator*{\A}{\mathcal{A}}
\DeclareMathOperator*{\dist}{dist}
\DeclareMathOperator*{\diam}{diam}
\DeclareMathOperator*{\BV}{BV}

\theoremstyle{plain}
\newtheorem{definition}{Definition}[section]
\newtheorem{example}[definition]{Example}
\newtheorem*{theorem*}{Theorem}
\newtheorem*{conjecture*}{Conjecture}
\newtheorem*{example*}{Example}
\newtheorem*{definition*}{Definition}
\newtheorem*{remark*}{Remark}
\newtheorem*{prop*}{Proposition}
\newtheorem{theorem}[definition]{Theorem}

\newtheorem{proposition}[definition]{Proposition}
\newtheorem{lemma}[definition]{Lemma}
\newtheorem{remark}{Remark}
\newtheorem*{lemma*}{Lemma}

\usepackage{pgfplots}
\pgfplotsset{compat=newest}
\usetikzlibrary{intersections, pgfplots.fillbetween}
\pgfdeclarelayer{pre main}
\allowdisplaybreaks

\providecommand{\keywords}[1]
{
  \small	
  \textbf{\textit{Keywords:}} #1
}

\usepackage{lmodern}
\usepackage{tcolorbox}

\begin{document}

\maketitle

\begin{abstract}
In this paper, we introduce a new concept of glued manifolds and investigate under which conditions the canonical heat flow on these glued manifolds is ergodic and irreducible. Glued manifolds are metric spaces consisting of manifolds of varying dimension connected by a weakly doubling measure. This can be seen as a condition on the jump in dimension. From another perspective, this construction also defines the Brownian motion on these glued spaces. Using the heat flow, we construct a nonlocal perimeter functional, the heat excess, to raise the question of its $\Gamma$-convergence to the standard perimeter functional.

In this context, we connect our work to the previous work on the convergence of perimeter functionals, approximations, and existence of heat kernels, as well as short-time expansions of Brownian motion.
\end{abstract}


\tableofcontents

\makeatletter
\providecommand\@dotsep{5}
\makeatother

\section{Motivation}

The heat equation is one of the most important and most often applied partial differential equations. This equation has applications in a large number of areas and fields. In this paper, we analyze the existence of an ergodic heat flow on so-called weighted glued manifolds.
These include spaces consisting of manifolds with different dimensions that are connected through weights. The idea behind these spaces is to incorporate a general setting, for instance, for manifold learning and similar applications where data is distributed locally in the shape of manifolds but can vary globally. This assumption is also called the union of manifold hypothesis. To be more precise, the data in machine learning is often given through points in a very high dimensional space. For instance, in the classical example of MNIST, a data set consisting of hand-drawn digits, one has a training data set of 60{.}000 digits which are represented by $28\times 28$ gray scale images. The canonical embedding space is $\mathbb{R}^{784}$. However, data can be classified by algorithms with comparatively few parameters. This success gives rise to the manifold hypothesis, namely that the data possesses an intrinsic lower dimensional structure. In the example of MNIST, one could think of a few parameters describing the structure of each digit (for $8$ this would be two parameterized loops). Here, one can already see that not all digits have the same variety. For example, a $0$ has less features than an $8$. The union of manifold hypothesis suggests that one should think of the data as combined manifolds of different dimensions. This is supported by numerical simulations~\cite{UnionManifold}.

The weights in our formulation ensure that each component contributes in the same way and that the data are overall continuous in a weak sense, which is made precise later. To state our main theorem, Theorem~\ref{mainThm}, in an informal way, we prove the following:
\begin{theorem*}
    Let $(X,d,\mu)$ be the union of weighted smooth manifolds with transversal intersections such that $\mu$ satisfies a weak doubling condition. Then, the canonical heat flow in $X$ is ergodic and irreducible.
\end{theorem*}

The heat flow here is defined through the Dirichlet form given from the manifolds with the combined volume measure. This definition leads to the Laplace-Beltrami operator on both manifolds individually and has a Dirichlet and Neumann compatibility condition at the intersection, that is, the inflow and outflow coincide, which ensures mass preservation. This also constructs the canonical Brownian motion as the process with the Laplace-Beltrami operator as the generator. The ergodicity and irreducibility follows from a spectral gap obtained for the Laplace-Beltrami operator.

The weak doubling condition heuristically says that the jump in dimension between the manifolds cannot be too high. If the effective difference, which is the difference with consideration of the weights, is larger than one, the flow will separate and the manifolds will not be connected.

Since the heat flow is ergodic and irreducible, it is possible to investigate the short-term behavior of the heat flow and ask the question of the $\Gamma$- (or Mosco-)convergence of a nonlocal perimeter functional to the local perimeter. With this in hand, one can construct a classifier on these spaces. Moreover, the definitions carry over to a more general setting of metric measure spaces instead of smooth compact manifolds. For more details, see Section~\ref{Outlook}.

In Section~\ref{background}, we compare our results with existing work and present related settings and papers. 

\subsection{Structure of the paper}

First, in Section~\ref{Definitions}, the setting in which we are working and the basic definitions are stated together with our main theorem in a precise form, Theorem~\ref{mainThm}. We consider a class of glued manifolds with mild assumptions on their weights to ensure connectedness in a weak sense.\\
Additionally, we recapitulate the standard theorems and properties which we apply later. The advanced theorems and notations are introduced in their respective sections to avoid overloading Section~\ref{Definitions}.

Afterwards, in Section~\ref{manifoldSetting}, we prove our main theorem in a simple setting in three steps. We begin with the definition of the heat flow and associated operators and spaces. Next, we prove continuity in a weak sense via capacity bounds and conclude by a compact embedding in the sense of Rellich-Kondrachov. This embedding yields the ergodicity and irreducibility.\\
Basically, the idea is to use the local structure of the manifolds to show that only constant functions remain static under the flow. Together with spectral theory of the resolvent associated with the flow, this allows us to prove a spectral gap and exponential convergence to constants for every initial condition. This convergence, in turn, implies ergodicity and irreducibility.

In Section~\ref{Generalization}, we point out ways to extend the results to more general weights for glued spaces and prove that our procedure from Section~\ref{manifoldSetting} can be iterated to the setting of multiple glued manifolds as long as they only intersect pairwise.

Next, in Section~\ref{Outlook}, we present the definitions and basic theory needed to state a conjecture about the $\Gamma$-convergence (or Mosco-convergence) of a nonlocal perimeter functional to the perimeter functional in glued spaces and general metric measure spaces. This functional is known as the heat excess and uses the heat flow constructed beforehand.

Finally, in Section~\ref{background}, we discuss related work in the different areas of this paper and connections to the different fields of study where these problems occur.

\section{Setting and Main Theorem}
\label{Definitions}

We now present the setting for the metric measure space $(X,d,\mu)$. We first state the results for two compact manifolds and then extend to a generalized setting.

Our domain $X$ is the union of two embedded smooth Riemannian manifolds that can vary in dimension and intersect. To be precise,
$$X\coloneqq M_1\cup M_2.$$
where $M_i$ are compact connected Riemannian manifolds with a smooth boundary. We assume that the dimension of $M_i$ is $n_i$ and the intersection $\emptyset\neq L\coloneqq M_1\cap M_2$ lies within $M_i$ and that $L$ is a smooth manifold of dimension $k\leq \min(n_1,n_2)$. This is the case if, for instance, the intersection is transversal in the surrounding space or a smooth manifold containing an open neighborhood of the intersection, see Definition~\ref{Def:Transversal}.

The distance $d$ is defined as the glued distance defined as follows, where $d_i$ is the distance in $M_i$.
\begin{definition}[Glued distance]
    For the manifolds $M_1,M_2$, we define the glued distance $d$ as:
    \begin{align*}
        d(x,y)&\coloneqq\inf\left\{\sum\limits_{i=0}^{N-1} \widetilde d(x_i,x_{i+1}):x_0=x, x_N=y\right\},\\
    \widetilde d(x,y)&\coloneqq\begin{cases}
        \min(d_1(x,y),d_2(x,y)) &\text{for }x, y\in L,\\
        d_1(x,y) &\text{for }(x,y)\in M_1\times M_1, (x,y)\not\in M_2\times M_2,\\
        d_2(x,y) &\text{for }(x,y)\not\in M_1\times M_1, (x,y)\in M_2\times M_2,\\
        \infty &\text{for }(x,y)\not\in (M_1\times M_1)\cup (M_2\times M_2).
    \end{cases}
    \end{align*}
\end{definition}
which is obtained by taking the infimum over the length of all connecting paths between points. This distance restricted to each manifold $M_i$ is equivalent to the original distance on the manifolds, since the intersection $L$ is a compact smooth submanifold. All metric balls $B$ in this chapter are considered with respect to this distance $d$ on $X$ if not stated otherwise. But we show that other metrics induce the same topology and even are equivalent, see the comment after Proposition~\ref{Prop:LipChar}.

Lastly, $\mu$ takes the role of a volume measure on $X$ that connects the manifolds and is the sum of the weighted volume measures, namely
$$\mu\coloneqq \omega_1\cdot\text{vol}_{M_1}+\omega_2\cdot\text{vol}_{M_2}$$
where the weights $\omega_i\in A_2(M_i)$, $\omega_i>0$ are Muckenhoupt, see Definition~\ref{Def:Muckenhoupt}. Additionally, we require that $\mu(L)=0$ and that the manifolds $M_i$ are transversal. Intuitively, this says that the manifolds do not coincide or go in the same direction in a local neighborhood near the intersection.

\begin{definition}[Transversality, {\cite[Section 1.5]{guillemin2010differential}}]
    Two manifolds $M_1,M_2$ are trans\-ver\-sal if 
    \begin{itemize}
        \item there exists a smooth manifold $M$ and a local neighborhood $U$ containing the intersection $L$ such that $M_i\cap U$ is a smooth submanifold of $M\cap U$ of dimension $n_i$
        \item for all $x\in L$, we have a relation of the tangent spaces:
        $$T_xM_1 + T_xM_2 =T_xM.$$
    \end{itemize}
    \label{Def:Transversal}
\end{definition}

Given this definition, there is a direct consequence which states the basic properties of the intersection.

\begin{proposition}[{\cite[Section 1.5]{guillemin2010differential}}]
    If $M_1$ and $M_2$ are transversal in $M$, their intersection $L$ is again a submanifold of $M$ with 
    $$\text{codim}(L)=\text{codim}(M_1) + \text{codim}(M_2)$$
    or equivalently, $k=n_1+n_2-n$.
\end{proposition}

Moreover, we assume that the measure $\mu$ satisfies a weak continuity in the style of a doubling condition:
\begin{definition}[$N$-doubling condition]
    A measure $\mu$ on a metric measure space $(X,d,\mu)$ is called $N$-doubling if for all $x\in X$ and radii $r>0$ the metric balls are doubling up to a function $N(r):$
    $$\mu(B_{3r}(x))\leq N(2r)\mu(B_r(x)).$$
\end{definition}
This ensures a weak type of continuity of $\mu$. This type of condition is the best we can hope for in this class of spaces, as a doubling condition would imply that the whole space has the same dimension. Moreover, there is a rich class of spaces that fulfill this type of generalized doubling condition. We provide examples in Example~\ref{Ex:1d2d} and Example~\ref{Ex:generalWeights}. 

\begin{remark*}
    The $N$-doubling condition can also be formulated with the Euclidean (extrinsic) distance. In this case, one would need to strengthen the transversality into a metric-quantitative version, i.e., there exist balls near the intersection which do not intersect the other manifold.
\end{remark*}

We note that $\mu$ is a positive Radon measure with $\supp \mu=X$ and that $L^2(X,\mu)$ is a Hilbert space.

\begin{remark*}
    Alternatively, one could consider another connection of the manifolds instead of the weights. It is possible to get a continuity over the intersection by collapsing an area around it into a lower-dimensional submanifold and changing the distances and topology accordingly. In this paper, we use the approach with weights as it satisfies a continuity over the whole domain and fits more naturally into the framework of manifold learning. With continuity, we mean that the effective dimension (including potentially vanishing or singular weights) does not jump and changes continuously. This ensures capacity bounds and trace conditions as proven in Section~\ref{manifoldSetting}.
\end{remark*}

Now, we are able to state the main theorem of this paper.

\begin{theorem}[Heat flow on the glued manifold]
\label{mainThm}
    Let $X=M_1\cup M_2$ with the glued distance $d$ be a given transversal union of two smooth compact manifolds with boundary such that intersection $\emptyset\neq L\coloneqq M_1\cap M_2$ lies in the inner of both of them with $\text{dim}(M_i)=n_i$, $\text{dim}(L)=k$. Moreover, assume nonnegative weights $\omega_1\in A_2(M_1)$ and $\omega_2\in A_2(M_2)$, $\omega_i>0$ such that $\omega_2$ is bounded from above and below, $\omega_1$ is bounded in a neighborhood of $\partial_1 M_1$ and $k=n_2-1$ where without loss of generality $n_1\geq n_2$.
    Additionally, assume that
    $$\mu\coloneqq \omega_1\cdot\text{vol}_{M_1}+\omega_2\cdot\text{vol}_{M_2}$$
    fulfills $\mu(L)=0$ and is $N$-doubling with $N\in L^1_{loc}(\mathbb{R}_{\geq 0})$.

    Then, the canonical heat flow is ergodic and irreducible.
\end{theorem}

This theorem can be generalized in a form that allows for unbounded weights on both manifolds and for multiple manifolds that intersect. This is done in Section~\ref{Generalization}, Theorem~\ref{Thm:GenMainThmHeatGlued}, and follows the same ideas and steps as the proof of Theorem~\ref{mainThm}. For clarity and to be better readable, we first state the simpler version and prove this.

The construction of the heat flow is done via its semigroup, which itself originates from the natural Dirichlet form $\mathcal{E}$ on $X$:
\begin{align*}
    \E(f)&\coloneqq \int\limits_X|\nabla f|^2\dd\mu\coloneqq \int\limits_{M_1}|\nabla^{M_1} f|^2\dd\mu_1+\int\limits_{M_2}|\nabla^{M_2} f|^2\dd\mu_2,\\
    D(\E)&\coloneqq \overline{C^\infty(\R^{\widetilde n})}^{\sqrt{\norm{\cdot}_{L^2(X,\mu)}^2+\E(\cdot)}}=:H^1_*.
\end{align*}
Here, $\widetilde n$ is the dimension of the surrounding space in which $X$ is embedded. 

That is, $\E$ is the manifold gradient in each manifold $M_i$ combined with the partial measure $\mu_i\coloneqq \omega_i\cdot\text{vol}_{M_i}$. Sometimes, this Dirichlet form is also called Cheeger energy; see, for example, \cite{buffa2021bv}. We call the individual Dirichlet forms $\E_i(f)\coloneqq \int_{M_i}|\nabla^{M_i} f|^2\dd\mu_i$. For a general definition of Dirichlet forms and their properties, see Definition~\ref{Def:Dirichletform} and the following.

\begin{figure}[!ht]
    \centering
    \begin{tikzpicture}
        \draw (0,0) -- (0,-2);
        \shade[shading=radial, inner color=white, outer color=gray!60!white, opacity=0.70] (2,0) arc (0:360:2 and 0.6); 
        \draw (2,0) arc (0:360:2 and 0.6) node[above right]{$M_1$};
        \draw (0,0) -- (0,3) node[right]{$M_2$};
        \draw[domain=-2:-1/15, smooth, variable=\x, gray, dashed] plot ({\x}, {abs(5*\x)^(-1)});
        \draw[domain=1/15:2, smooth, variable=\x, gray, dashed] plot ({\x}, {abs(5*\x)^(-1)}) node at (1.0, 1.2) {$\omega_1=\frac{1}{|x|}$};
    \end{tikzpicture}
    \caption{Example of a glued manifold consisting of the weighted manifolds $M_1$ and $M_2$.}
    \label{fig:example}
\end{figure}

\begin{example}
\label{Ex:1d2d}
    A typical example would be a one-dimensional line intersecting a two-dimensional disk with the density $\frac{1}{|x|}$ that lifts the dimensions at their intersection. This is done to ensure that their dimensions do not differ too much and that the measure is continuous in the sense of an $N$-doubling condition. 
    
    A visualization can be seen in Figure~\ref{fig:example}. We show that if the difference in dimension at their intersection is less than one, the heat flow connects both of them and is ergodic. This is an example of a transversal intersection as one can take $\mathbb{R}^3$ as the surrounding manifold and compute $\spn\{e_1,e_2\}+\spn\{e_3\}=\mathbb{R}^3$.
    
    In the case that the difference in their dimension is too large, the capacity of the intersection with respect to one of both manifolds is zero, and thus the heat flow does not need to be ergodic. For this precise example, the disk and line each without density, the flow would be separate on the disk and the line. This is due to the fact that the intersection is a point and thus has capacity zero in the disk, i.e., it is too small for the flow and the Brownian motion meets it with probability zero, see~\cite[Theorem 3.2]{BMCap}. Similarly, from the perspective of the line, the mass of the disk near the intersection is not enough.
\end{example}

Now, we give the precise definition of Dirichlet forms as well as of Muckenhoupt weights.
The following definitions can be compared to~\cite{fukushima2010dirichlet, ruiz2023dirichlet, sturm1994analysis}.
\begin{definition}[Dirichlet form]
\label{Def:Dirichletform}
A symmetric form (nonnegative definite densely defined symmetric bilinear form) $(\mathcal{E}, D(\mathcal{E}))$ on $L^2$ is called a Dirichlet form if it is closed and Markovian, i.e.,
$$D(\mathcal{E}) \text{ with norm }\norm{u}_{H^1_*}^2\coloneqq \norm{u}_{L^2}^2+\mathcal{E}(u,u) \text{ is complete}$$
and for
$$u\in D(\mathcal{E}), v=(0\lor u)\land 1,\text{ we have }v\in D(\mathcal{E}), \mathcal{E}(v,v)\leq \mathcal{E}(u,u).$$
A core $\mathcal{C}$ of $\mathcal{E}$ is a dense subset of $D(\mathcal{E})\cap C^0$ in the sense that $\mathcal{C}$ is dense in the domain $D(\mathcal{E})$ with respect to $\norm{u}_{H^1_*}$ and in $C^0$ with respect to $L^\infty$-norm. We call $\mathcal{E}$ regular if a core exists.

Lastly, a Dirichlet form $\mathcal{E}$ is called strongly local if for $u,v\in D(\mathcal{E})$ compactly supported and $u$ constant in a neighborhood of $v$: $\mathcal{E}(u,v)=0$. Furthermore, it is called strictly local (sometimes also called strongly regular) if the associated energy metric $d_\mathcal{E}$ (see below) is a metric and induces the same topology as $d$.
\end{definition}
\begin{remark*}
    We give a reason for the notation $\norm{u}_{H^1_*}$ in Section~\ref{manifoldSetting}.
\end{remark*}

We write $\mathcal{E}(u)$ for $\mathcal{E}(u,u)$. For closed forms, the Markovian condition can be rephrased as $\mathcal{E}$ being nonincreasing under all normal contractions.

Now, we define the energy metric and what we mean when we call a function (locally) Lipschitz on this space.
For $\varphi\in \mathcal{C}(\E)$, we define for $u$ in the core and thus by extension on the domain, the energy measure $\Gamma(u)$ which is a Radon measure, see~\cite{ruiz2023dirichlet, sturm1994analysis}, via:
$$\int_X \varphi\dd\Gamma(u)\coloneqq \mathcal{E}(\varphi u,u)-\frac{1}{2}\mathcal{E}(u^2,\varphi).$$
If $\Gamma$ is absolutely continuous, this can also be seen as a carr\'e-du-champ-operator if ones looks at the bilinear form obtained using a polarization formula and for $u,v,\varphi \in D(\mathcal{E})$
$$\int_X \varphi\dd\Gamma(u,v)=\frac{1}{2}\left(\mathcal{E}(\varphi u,v)+\mathcal{E}(\varphi v,u)-\mathcal{E}(\varphi,uv)\right).$$
Given $\Gamma$, we now can define the energy metric which sometimes is also called intrinsic metric for a regular Dirichlet form:
$$d_\mathcal{E}(x,y)\coloneqq \sup\left\{u(x)-u(y): u\in C^0(M), \frac{\dd \Gamma(u)}{\dd\mu}\leq 1\right\}.$$
This notion means that $\Gamma(u)$ is absolutely continuous with respect to $\mu$ and that the Radon-Nikodym derivative is bounded by $1$. A priori, this energy metric does not have to be a metric, as it can be infinite or zero for different $x,y\in M$. But for strictly local Dirichlet forms, this is by definition a metric.

We note that the forms $\mathcal{E}_i$ are strictly local regular Dirichlet forms, as they are the usual Dirichlet forms on smooth compact manifolds.

Finally, we call a function $u\in D(\mathcal{E})$ $\E$-Lipschitz, $u\in \Lip_{\E}(X)$, if $\Gamma(u)\ll\mu$ and the Radon-Nikodym derivative is bounded. This bound is then called the Lipschitz constant. In this case, we write $|\nabla u|^2\coloneqq \frac{\dd \Gamma(u)}{\dd\mu}$ and $|\nabla u|\coloneqq \sqrt{|\nabla u|^2}$. This coincides with our intuition that $\mathcal{E}$ is the $H^1$-seminorm and $\Gamma$ the square of the derivative. Moreover, for a strictly local Dirichlet form, this is equivalent to the usual definition of Lipschitz function with the distance $d_\mathcal{E}$ as seen in the next proposition. 
\begin{proposition}[Characterization Lipschitz functions, {\cite[Theorem 3.2, 4.1]{LipChar}}]
\label{Prop:LipChar}
    Let $X$ be a locally compact space with strictly local Dirichlet form $\mathcal{E}$. Then $(X,d_\mathcal{E})$ is a length space, and the space of Lipschitz functions $\Lip_{\E}$ in the above definition is a sheaf and coincides with the definition through the distance $d_\mathcal{E}$.
\end{proposition}
\begin{remark*}
    In the case where $\E$ is strictly local, we can work with the energy metric $d_\mathcal{E}$ instead of $d$. This is justified by the above proposition. By transversality and compactness, and since spaces are length spaces, this metric restricted to the partial spaces $M_i$ is equivalent to the original metrics $d_i$.
\end{remark*}

With the definition of our Dirichlet form $\E$ above with its domain $D(\E)$, we can see that $\E$ is also a strictly local regular Dirichlet form. The basic properties are proven in~\cite[Theorem 2.2]{paulik2005gluing} and the fact that it is strictly local in~\cite[Section 2.3]{paulik2005gluing}. Moreover, in this section it is also proven that the energy metric is equivalent to the glued metric. We could also define the domain of $\E$ via $\Lip_{\E}$ instead of $C^\infty$, which produces the same domain.

For the connectedness of $(X,d,\mu)$ in our sense, we need a Poincar\'e inequality.

\begin{definition}[$2$-Poincar\'e inequality]
    \label{Def:2PI}
    We say that $X$ satisfies a $2$-Poincar\'e inequality if there exists $c>0$ and a dilation factor $\lambda\geq 1$ such that for all integrable functions $f$ and for each ball $B\subseteq X$:
    \begin{align*}
        \fint\limits_B |f-f_B|^2\dd\mu&\leq cr^2\fint\limits_{\lambda B} |\nabla f|^2\dd\mu.
    \end{align*}
    Here, $f_B$ is the mean of $f$ over $B$:
    $$f_B\coloneqq \fint\limits_B f\dd\mu=\frac{1}{\mu(B)}\int\limits_B f\dd\mu.$$
\end{definition}
\begin{remark}
    \label{Rem:PIconn}
    For generalizations of the definitions, see Section~\ref{Outlook}.

    Note that if $X$ satisfies a Poincar\'e inequality, it has to be connected up to negligible sets. Assume that it would not be connected up to a null set and let $X_1$ be one component with $f=\chi_{X_1}$ and $|\nabla f|=0$. Then, this would be a contradiction as $\mu(X_1)>0$ and $1>f_B>0$ for $B$ large enough.

    If not specified, we assume the dilation factor $\lambda$ to be one.
\end{remark}

The heat flow is defined using its semigroup, which originates in the Dirichlet form.
In summary, this can be done as follows. Given the Dirichlet form $\E$, we define the Laplace-Beltrami operator $\A$ with its domain $D(\A)$ in a way that is reminiscent of the weak Laplacian. For details, see Subsection~\ref{Domain}. We show that $\A$ is a selfadjoint symmetric nonpositive operator that induces an analytic semigroup $S_h$, which we call the heat semigroup.

In Section~\ref{manifoldSetting}, we also show the continuity and ergodicity properties of these objects. For this, we need capacity bounds.

\begin{definition}[Capacities]
\label{Def:Capacity}
    Let $\Omega$ be a (relative) open subset of $X$. We define the relative $(2,\mu)$-capacity on compact sets $K$.
    \begin{align*}
        {\Capp}_{2,\mu}(K,\Omega)\coloneqq \inf\left\{\int\limits_\Omega |\nabla u|^2\dd\mu: u\in \Lip_{\E}(X)\cap C_c(\Omega), u\geq \chi_K\right\}.
    \end{align*}
    This definition then extends to measurable sets by approximation through inner regularity.
\end{definition}

\begin{remark}
\label{Rem:CapXi}
    In the case of a glued space, we write ${\Capp}_{2,\mu_i}(K,\Omega)$ as a shorthand for ${\Capp}_{2,\mu_i}(K\cap M_i, \Omega\cap M_i)$.
\end{remark}

For an overview of the theory of capacities and more general definitions, see~\cite{heinonen2018nonlinear, kilpelainen1994weighted, kinnunen1996sobolev}.

Additionally, we show that our constructed heat flow connects the entire space. The right notion to make this precise is the notion of ergodicity and irreducibility, see~\cite{fukushima2010dirichlet}. In the following, let $S_h:L^2(X,\mu)\to L^2(X,\mu)$ be a semigroup.

\begin{definition}[Irreducibility]
    \label{def:irreducible}
    We say that a measurable set $E\subseteq X$ is $S_h$ invariant if for all $f\in L^2(X,\mu)$ and $h>0$, we have 
    $$S_h(\chi_Ef)=\chi_ES_h(f)\ \mu\text{-a.e.}$$ 
    
    Then, $S_h$ is called irreducible if for all $S_h$ invariant sets $E$, we have $\mu(E)=0$ or $\mu(E^c)=0$.
\end{definition}

\begin{definition}[Ergodicity]
    \label{def:ergodic}
    An $L^2(X,\mu)$ semigroup $S_t$ is ergodic if the set function $T_t$ defined by
    $$T_t:X\to X,\quad E\mapsto \supp(S_t\chi_E)$$
    is ergodic for each fixed $t>0$, that is, if for a measurable set $E$ with $T_t^{-1}(E)= E$, we have $\mu(E)=0$ or $\mu(E^c)=0$.
\end{definition}

Lastly, for our setting of glued manifolds with weights, we need to define what is means to be a Muckenhoupt weight.

\begin{definition}[Muckenhoupt weights]
    \label{Def:Muckenhoupt}
    We say $\omega$ is in the class $A_2(M)$ of $2$-Muck\-en\-houpt weights on a smooth compact manifold $M$, if for all balls $B$,
    $$\fint\limits_B \omega_i\dd\vol_{M_i}\cdot \fint\limits_B \frac{1}{\omega_i}\dd\vol_{M_i}\lesssim 1.$$
    Here, $\lesssim$ mean that the left hand-side is bounded by the right hand-side up to a universal constant.
\end{definition}

\begin{remark*}
    For the typical weight $\omega=|x|^\alpha$ in $\mathbb{R}^n$, we have $\omega\in A_2$ iff $\alpha\in (-n,n)$.
\end{remark*}

\subsection{Basic Theorems for \texorpdfstring{$A_2$-}{A2-}Weights on Manifolds}
\label{SS:A2Prop}

In this section, we show that, if we have a smooth compact manifold $M$ and a Muckenhoupt weight $\omega\in A_2(M)$, then the measure $\mu=\omega\cdot{\vol}_M$ is doubling and fulfills the $2$-Poincar\'e inequality and a reverse Hölder inequality.
These theorems are mostly due to \cite[Chapter 15]{heinonen2018nonlinear} where they are carried out in the case of a flat Euclidean space. But one can also work in charts to obtain them on smooth manifolds.
\begin{lemma}[$A_2$ doubling]
A Muckenhoupt weight satisfies a volume doubling condition:
$$\mu(B_{2r})\lesssim \mu(B_r)\quad\forall r>0.$$
\end{lemma}
\begin{proof}
    For $E\subseteq B$ where $B$ is a ball, we have:
    \begin{align*}
        \vol_M(E)&=\int\limits_E \frac{\omega^\frac{1}{2}}{\omega^\frac{1}{2}}\dd x\\
        &\leq \sqrt{\mu(E)}\sqrt{\int\limits_E \frac{1}{\omega}\dd x}\\
        &\leq \sqrt{\mu(E)\vol(B)}\sqrt{\fint\limits_B \frac{1}{\omega}\dd x}\\
        &\lesssim  \sqrt{\mu(E)}\vol(B)\sqrt{\frac{1}{\mu(B)}}\\
        \Rightarrow \mu(B)&\lesssim \left(\frac{\vol_M(B)}{\vol_M(E)}\right)^2\mu(E).
    \end{align*}
    A direct consequence is the volume doubling for $E=B_r, B=B_{2r}$ since the volume measure of a compact manifold is doubling:
    \begin{align*}
        \mu(B_{2r})\lesssim \mu(B_r).
    \end{align*}
\end{proof}

The other properties can be proven locally on charts as we are only interested in inequalities on balls with a radius bounded above by an arbitrary but fixed bound.

For this, take balls $B_r^M$ in $M$. They are transformed by transition map $\phi$ to $A_r$. Since $\phi$ can be chosen to be Lipschitz, there exists $c\geq 1$ such that $B_\frac{r}{c}\subseteq A_r\subseteq B_{cr}$ and thus, since the metric is bounded, 
$$\vol_M(B_r^M)\simeq |A_r|\simeq |B_r|.$$
Therefore,
\begin{align*}
    \int\limits_{B_r}\omega\dd x\int\limits_{B_r}\frac{1}{\omega}\dd x&\leq \int\limits_{A_{cr}}\omega\dd x\int\limits_{A_{cr}}\frac{1}{\omega}\dd x\\
    &\lesssim |A_{cr}|^2\lesssim |B_r|^2.
\end{align*}
We conclude that $\omega$ (to be precise, $\omega\circ \phi^{-1}$) is also a Muckenhoupt weight on the flat chart, similar for $\omega\in A_p$, see Section~\ref{Outlook}. Hence, we can follow the usual proofs of \cite[Chapter 15]{heinonen2018nonlinear} to obtain a $p$-Poincar\'e inequality without dilation, cf.~Definition~\ref{Def:pPoincar\'e}. First, we can prove the doubling condition and also the reverse doubling condition which states that for each ball $B$ and measurable subset $E\subseteq B$ there exists $0<q<1$ such that
$$\frac{\mu(E)}{\mu(B)}\leq \left(\frac{|E|}{|B|}\right)^q.$$
From this, we directly conclude the reverse Hölder inequality: There is $r>1$ such that 
$$\left(\fint_B \omega^r\dd x\right)^\frac{1}{r}\lesssim \fint_B\omega\dd x.$$
This inequality can be used to obtain the open end property, that is, if $\omega\in A_p$ with $p>1$ then $\omega\in A_{p-\varepsilon}$ for some $\varepsilon>0$. Concluding, this yields an equivalent characterization via the continuity of the Hardy-Littlewood maximal function on $L^p(\mu)$-functions, which finally proves a Poincar\'e inequality on local balls for $
\mu=\omega\cdot{\vol}_M$ with $\omega\in A_2(M)$:
$$\int\limits_{B_r}|f-f_{B_r}|^2\dd \mu\lesssim r^2\int\limits_{B_r}|\nabla f|^2\dd \mu.$$

It is noteworthy that all consideration could also be done from the perspective of Brownian motion instead of the heat flow as they are equivalent.

In this paper, we focus on the perspective of the heat flow. The Brownian motion is a Markov process with the Laplace-Beltrami operator as infinitesimal generator. More general, we define the Brownian motion on a locally compact metric measure space with strongly local Dirichlet form as the Markov process with the semigroup $S_h$. This semigroup is constructed from the Laplace operator which comes from the Dirichlet form as explained in this and the next section; cf.~\cite{fukushima2010dirichlet}.

The heat flow and the Brownian motion are closely related to each other. In $\mathbb{R}^n$ or smooth manifolds, one could think of the heat kernel as the transition density of the Brownian motion. Similarly, we can solve the heat flow using the Brownian motion as the evaluation points, i.e., $u(x,t)=\mathbb{E}_x[u(B_t,0)]$. An example of a sample path of the Brownian motion on the sphere can be seen in Figure~\ref{Fig:BMSphere}.

\begin{figure}
    \centering
    \includegraphics[width=0.4\linewidth]{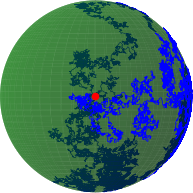}
    \caption{A sample path of the Brownian motion on a sphere.}
    \label{Fig:BMSphere}
\end{figure}

\section{Proof of Main Theorem}
\label{manifoldSetting}

\subsection{Definition of the Heat Flow}
\label{Domain} 

Given the setting of Section~\ref{Definitions}, we prove our main theorem, Theorem~\ref{mainThm}. This theorem can be extended to the more general case; for this, see Section~\ref{Generalization}. This includes the extension to multiple manifolds glued together with general Muckenhoupt weights on them.

The proof consists of three different steps to prove properties of the resulting space as well as of the operators and semigroup on it. This is done in the following. First, we define the relevant operators and function spaces. Afterwards, we prove capacity bounds. This we use to show a Rellich-Kondrachov embedding which yields ergodicity and irreducibility.

We start with the definition of the heat flow via a given Dirichlet form using the Laplace-Beltrami operator.

We recall from Section~\ref{Definitions} the energy norm
$$\norm{f}_{H^1_*}\coloneqq \sqrt{\norm{f}_{L^2(X,\mu)}^2+\E(f)}$$
and the domain of $\mathcal{E}$
$$D(\E)\coloneqq \overline{C^\infty(\R^{\widetilde n})}^{\norm{\cdot}_{H^1_*}}=:H^1_*.$$
Here, $\widetilde n$ is the dimension of the surrounding space $X$ is embedded in and we associate functions if their difference has norm zero. The star symbolizes the continuity which we see in Section~\ref{Capacity}. Heuristically, this can be seen as a pointwise continuity on the intersection and Neumann conditions at the boundary as well as flow conditions at points where the manifolds meet. As discussed in Section~\ref{Definitions}, one can also choose an intrinsic definition via $\Lip_{\E}$.

Additionally, we define the Laplace-Beltrami operator $\A$ with its domain $D(\A)$ using the Dirichlet form in the following way
$$D(\A)\coloneqq \{f\in D(\E)|\exists \A f\in L^2(X,\mu): (\A f,g)=-\E(f,g)\;\forall g\in D(\E)\}.$$
We note that $\E$ is densely defined, symmetric, positive, continuous with respect to $\norm{\cdot}_{H^1_*}$ (following from $\E_i$ by definition of $\E$) and closed (by Definition of $D(\E)$), compare also \cite[Theorem 2.2]{paulik2005gluing}, which implies, for instance by~\cite[Proposition 1.22]{ouhabaz2009analysis}, that $\A$ is a densely defined selfadjoint symmetric nonpositive operator and $\Id-\A$ is invertible. It induces an analytic semigroup $S_h$ that we call the heat semigroup. The spectral properties of $\A$ are computed via the resolvent 
$$R_{\A} =(\Id-\A)^{-1}:L^2(X,\mu)\to D(\A)\subseteq D(\E)$$
which satisfies $(f,g)=(R_{\A} f,g)+\E(R_Af,g), Af=f-R_{\A}^{-1}f$ for $f\in D(\A)$.

In Section~\ref{Embedding}, we show and use that $R_{\A}$ is compact to prove a spectral gap on $\A$.

\subsection{Capacity Bounds}
\label{Capacity}
In this part, we show that the $N$-doubling condition implies a capacity equivalence if $N$ is integrable:

\begin{lemma}[Capacity equivalence]
    Given the assumptions of Theorem~\ref{mainThm}, the capacity of the intersection is positive in both manifolds, that is, there exists $R>0$ such that 
    $${\Capp}_{2,\mu_1}(L,L_R)>0,\quad {\Capp}_{2,\mu_2}(L,L_R)>0$$
    where $L_R\coloneqq L+B_R=\bigcup\limits_{x\in L} B_R(x)$.
\end{lemma}

The proof is given in this subsection.
We recall that the intersection of the manifolds $L$ is a smooth manifold of dimension $k$.

In the following, we treat the intersection as a linear subspace $L\subseteq \R^k$. This is justified by considering charts of the manifolds such that they agree on the intersection $L$ and possess smooth (bi-Lipschitz) transition maps that map $L$ into the linear subspace.
By compactness of $L$ and $M_i$, finitely many suffice and the capacity of $L$ in $M_i$ is finite iff there is a chart in which the relative capacity is finite.

Moreover, by the continuity of the metric, the charts are chosen such that $\sqrt{\det g}$ is uniformly bounded from above and below and so that they agree on $L$ for both manifolds. Thus, in this section, we work with the Lebesgue measure $\mathcal{L}^n=\dd x$ instead of the volume measure ${\vol}_M$ because the calculations are equivalent up to a factor.

First, we need an equivalence of the partial measures close to the intersection.

Let $R_*(x)$ be small enough so that for every $x_*\in L, R<R_*$ there is a point $x\in \partial B_\frac{R}{2}(x_*)\cap M_1$ with $B_\frac{R}{2}(x)\subseteq M_1$. Then, we can estimate the partial measure $\mu_i$ on the balls at the intersection. Without loss of generality, only written down in the direction $"M_2\leq M_1"$:
\begin{align*}
    \mu_2(B_R(x_*))&\leq\mu(B_R(x_*))\leq \mu(B_{\frac{3}{2}R}(x))\leq N(R)\mu(B_\frac{R}{2}(x)))\\
    &=N(R)\mu_1(B_\frac{R}{2}(x))\leq N(R)\mu_1(B_R(x_*)).
\end{align*}

This means that we have
$$\forall x\in L, R<R_*(x): N(R)^{-1}\leq\frac{\mu_1(B_R(x))}{\mu_2(B_R(x))}\leq N(R).$$
\begin{remark*}
    $R_*(x)$ can be uniformly bounded from below, since $M_i$ and thus also $L$ are compact manifolds.

    The point $x$ must exist by dimensional arguments as $k<k+1=n_2<n_1$.
    Therefore, there exists a direction solely in $T_{x_*}M_i$ for each $x_*\in L$ that yields the existence of $x$ by the smoothness of the manifolds.
\end{remark*}

To compute the capacity, it proves useful to consider a thickened version of the intersection $L_R\coloneqq\bigcup\limits_{x\in L} B_R(x)$ as we estimate the relative capacity ${\Capp}_{2,\mu}(L_r,L_R)$ on $L_R$ and use an approximation $L_r\to L$ as $r\searrow 0$. In the following, we estimate the capacity below for both manifolds to ensure that the manifolds are sufficiently connected. The techniques are based on \cite[Chapter 2]{heinonen2018nonlinear} where the case of balls is proven.

For this calculation, we need a representation formula. We work on $\mathbb{R}^{n_i}$. By the estimate of the volume measure and the smoothness of the transition map, this is justified, as the calculations carry over to the charts and manifolds. Fix a function $u\in C^1_c$. Then, for a direction $w\in\mathbb{R}^{n_i}$ such that $|w|=1$:
$$u(x)=-\int\limits_0^\infty \partial_r u(x+rw)\dd r.$$
The procedure can be seen geometrically as cylinder like coordinates, see Figure~\ref{fig:procedure}.
By integrating over the unit sphere in dimension $n_i-k$ (extended with $0$ components in $\mathbb{R}^{n_i}$) and with $x=0$ and $u(0)=1$, we get the following:
\begin{align*}
    c_{n_i-k}&=-\int\limits_0^\infty\int\limits_{\mathbb{S}^{n_i-k-1}}\partial_ru(rw)\dd\mathcal{H}^{n_i-k-1}(w)\dd r\\
    &=-\int\limits_{\mathbb{R}^{n_i-k}}r^{-(n_i-k-1)}\partial_ru(rw)\dd\mathcal{H}^{n_i-k}(r\cdot w)\\
    &=-\int\limits_{\mathbb{R}^{n_i-k}}\frac{r}{|rw|^{n_i-k}}w\cdot\nabla^{n_i-k}u(rw)\dd\mathcal{H}^{n_i-k}(r\cdot w)\\
    &=-\int\limits_{\mathbb{R}^{n_i-k}}\frac{y\cdot \nabla^{n_i-k}u(y)}{|y|^{n_i-k}}\dd\mathcal{H}^{n_i-k}(y).
\end{align*}
This allows an estimate for the capacity of partial radially symmetric sets for $u$ as in the definition of ${\Capp}_{2,\mu}(L_r, L_R)$ (Definition~\ref{Def:Capacity}):
\begin{align*}
    1&\simeq \left(\int\limits_L \dd\mathcal{H}^k(x)\right)^2\\
    &\simeq\left(\int\limits_L \int\limits_{\mathbb{R}^{n_i-k}(x)}\frac{(x-y)\cdot \nabla^{n_i-k}u(y)}{|x-y|^{n_i-k}}\dd\mathcal{H}^{n_i-k}(y)\dd\mathcal{H}^k(x)\right)^2\\
    &\lesssim \int\limits_{L_R}|\nabla^{n_i-k}u|^2\dd\mu_i\cdot \int\limits_{L_R\setminus L_r}\frac{1}{\dist(y, L)^{2(n_i-k-1)}}\frac{\dd\LL^{n_i}(y)}{\omega_i(y)}.
\end{align*}
By estimating the gradient pointwise against the lower dimensional gradient, this yields:
\begin{align*}
    {\Capp}_{2,\mu_i}(L, L_R)&={\Capp}_{2,\mu_i}\left(\bigcap\limits_{r> 0}L_r, L_R\right)=\lim\limits_{r\searrow 0} {\Capp}_{2,\mu_i}(L_r, L_R)\\
    &\geq \lim\limits_{r\searrow 0}\inf\left\{\int\limits_{L_R} |\nabla^{n_i-k} u|^2\dd\mu_i:u\geq \chi_{L_r}\right\}\\
    &\gtrsim \left(\int\limits_{L_R}\frac{1}{\dist(y, L)^{2(n_i-k-1)}}\frac{\dd\LL^{n_i}(y)}{\omega_i(y)}\right)^{-1}.
\end{align*}

\begin{figure}[!t]
    \centering
    \begin{tikzpicture}[z ={(0,0,-cos(45))}, scale=2]
        \draw (0,-1.2,0)--(0,0,0);
        \coordinate (A) at (-1,0,-1); 
        \coordinate (B) at (1,0,-1); 
        \coordinate (C) at (1,0,1); 
        \coordinate (D) at (-1,0,1); 
        \draw (A) -- (B) -- (C) node[pos=0.6, right]{$\mathbb{R}^{n_i-k}$} -- (D) -- cycle;
        \shade[shading=radial, inner color=white, outer color=gray!60!white, opacity=0.70] (A) -- (B) -- (C) -- (D) -- cycle;
        \draw (0.5,0,0) arc (0:360:0.5 and 0.15) node[above right=-0.1]{$\mathbb{S}^{n_i-k}$};
        \draw[dotted] (0.5,0.5,0) arc (0:360:0.5 and 0.15);
        \draw[dotted] (0.5,-0.5,0) arc (0:360:0.5 and 0.15);
        \draw[dotted] (0.5,-1,0) arc (0:360:0.5 and 0.15);
        \draw (0,0,0)--(0,0.7,0) node[above right]{$L\subseteq \mathbb{R}^k$};
    \end{tikzpicture}
    \caption{We prove the lower capacity bound via a foliation procedure as shown in this figure.}
    \label{fig:procedure}
\end{figure}

Next, we note that in our setting the density $\omega_2$ is chosen in such a way that a $L$-Muckenhoupt type condition on the intersection is satisfied.

\begin{definition}[$L$-Muckenhoupt condition]
\label{Def:LMucken}
    We say that $\mu=\omega\cdot {\vol}_M$ satisfies the $L$-Muckenhoupt type condition for a submanifold $L$ in $M$ with $\text{dim}(L)=k, \text{dim}(M)=n$, if:
    \begin{align*}
        \int_{L_R}\omega\dd\vol_{M}\int_{L_R}\frac{1}{\omega}\dd\vol_{M}\simeq R^{2(n-k)}.
    \end{align*}
\end{definition}

This is due to the fact that $\omega_2$ is bounded from above and below. This can be compared to a Muckenhoupt type condition on the level of $L$.

The lower bound can now be simplified into a more suitable version and connected to the other manifold. By the layer cake formula and the transformation $\rho=t^{-\frac{1}{2(n_i-k-1)}}$:
\begin{align*}
    \int\limits_{L_R}&\frac{1}{\dist(y,L)^{2(n_i-k-1)}}\frac{\dd\LL^{n_i}}{\omega}\\
    &=\int\limits_0^\infty \frac{\LL^{n_i}}{\omega}(x\in L_R:\dist(x,L)^{-2(n_i-k-1)}>t)\dd t\\
    &=R^{-2(n_i-k-1)}\frac{\LL^{n_i}}{\omega}(L_R)+\int\limits_{R^{-2(n_i-k-1)}}^\infty \frac{\LL^{n_i}}{\omega}\left(L_{t^{-\frac{1}{2(n_i-k-1)}}}\right)\dd t\\
    &=R^{-2(n_i-k-1)}\frac{\LL^{n_i}}{\omega}(L_R)+2(n_i-k-1)\int\limits_0^R \rho^{-(2n_i-2k-1)}\frac{\LL^{n_i}}{\omega}(L_\rho)\dd \rho.
\end{align*}

To combine the estimates on both manifolds, we now compare $\frac{\LL^{n_i}}{\omega}(L_R)$  on both manifolds to get an estimate of $\frac{\LL^{n_1}}{\omega_1}(L_R)$ by $\frac{\LL^{n_2}}{\omega_2}(L_R)$. To this end, consider the Vitali covering theorem with the covering of $L_R$ by $B_R(x)$ for $x\in L$. There exist finitely many $x_j\in L$ such that $B_R(x_j)$ are disjoint and $$L_R\subseteq \bigcup 5B_j$$
where $5B_j\coloneqq B_{5R}(x_j)$. Since $\omega_i$ is a 2-Muckenhoupt weight, we know that $\omega_i\LL^{n_i}(B_R)$ and $\frac{\LL^{n_i}}{\omega_i}(B_R)$ can be compared: 
$$\omega_i\LL^{n_i}(B_R)\cdot \frac{\LL^{n_i}}{\omega_i}(B_R)\simeq R^{2n_i}.$$

Hence,
\begin{align*}
    \frac{\LL^{n_1}}{\omega_1}(L_R)&\leq \sum\limits_j \frac{\LL^{n_1}}{\omega_1}(5B_j)\lesssim \sum\limits_j \frac{\LL^{n_1}}{\omega_1}(B_j)\lesssim R^{2n_1}\sum\limits_j(\omega_1\LL^{n_1}(B_j))^{-1}\\
    &\leq N(R)R^{2n_1}\sum\limits_j (\omega_2\LL^{n_2}(B_j))^{-1}\lesssim N(R)R^{2n_1-2n_2}\sum\limits_j \frac{\LL^{n_2}}{\omega_2}(B_j)\\
    &\leq N(R)R^{2n_1-2n_2}\frac{\LL^{n_2}}{\omega_2}(L_R).
\end{align*}
In these estimates, we used the doubling condition on the weighted manifold $M_1$.

Together, this yields the following:
\begin{align*}
    &{\Capp}_{2,\mu_1}(L, L_R)\\
    &\gtrsim \left(R^{-2(n_1-k-1)}\frac{\LL^{n_1}}{\omega_1}(L_R)+2(n_1-k-1)\int\limits_0^R \rho^{-(2n_1-2k-1)}\frac{\LL^{n_1}}{\omega_1}(L_\rho)\dd \rho\right)^{-1}\\
    &\gtrsim \left(N(R)R^{-2(n_2-k-1)}\frac{\LL^{n_2}}{\omega_2}(L_R)+c\int\limits_0^R N(\rho)\rho^{-(2n_2-2k-1)}\frac{\LL^{n_2}}{\omega_2}(L_\rho)\dd \rho\right)^{-1}\\
    &\gtrsim \left(N(R)R^2\mu_2(L_R)^{-1}+c\int\limits_0^R N(\rho)\rho^{-(2n_2-2k-1)}\frac{\LL^{n_i}}{\omega_2}(L_\rho)\dd \rho\right)^{-1}.
\end{align*}

With our generalized Muckenhoupt condition on $\omega_2$, this has a more direct form and we can close the chain of arguments directly:
\begin{align*}
    {\Capp}_{2,\mu_1}(L, L_R)&\gtrsim \left(N(R)R^2\mu_2(L_R)^{-1}+c\int\limits_0^R N(\rho)\rho\mu_2(L_\rho)^{-1}\dd \rho\right)^{-1},\\
    {\Capp}_{2,\mu_2}(L, L_R)&\gtrsim \left(R^2\mu_2(L_R)^{-1}+2(n_2-k-1)\int\limits_0^R \rho\mu_2(L_\rho)^{-1}\dd \rho\right)^{-1}.
\end{align*}

Since we are in the case where $\omega_2$ is bounded and $k=n_2-1$, we see that the relevant condition on $N(R)$ is 
$$\int\limits_0^R N(\rho)\dd \rho<\infty$$
as we fixed $R>0$ and $\mu_2(L_\rho)\simeq \rho$. The dimensional condition on $k$ for a bounded weight is necessary for $\mu_2(L)=0$ and ${\Capp}_{2,\mu_2}(L, L_R)>0$.

Thus, the intersection has a finite positive capacity in $M_1$ in this setting if the same holds for $M_2$ and $\mu$ satisfies an $N$ double condition with integrable $N$.

Since the global and restricted capacity are equivalent for fixed $R>0$, this completes the proof.

\begin{remark}
    \label{Rem:Trace}
    Moreover, this implies that q.e.\ point of the intersection is a Lebesgue point, cf.~\cite[Theorem 3.5]{kilpelainen1994weighted} and \cite[Corollary 3.7]{kinnunen1996sobolev}. Therefore, we have a trace in the sense of a quasi-continuous refinement on $L$.
\end{remark}

\subsection{Compact Embedding of \texorpdfstring{$H^1_*$}{H1}}
\label{Embedding}

In this part, we show a compact Sobolev (or Rellich-Kondrachov) embedding of $H^1_*\hookdoubleheadrightarrow L^2(X,\mu)$.

We prove the compactness of the embedding through sequentially compactness and an inner approximation up to the boundary. The idea of this proof is based on the techniques of~\cite[Chapter 5]{paulik2005gluing}. For the convenience of the reader, we explain the details in our case here. 

\begin{lemma}[Compact embedding of $H^1_*$]
\label{Lemma:CompactEmbedding}
    Let $(X,d,\mu)$ be as in Theorem~\ref{mainThm} and $\E$ as before. Then, $H^1_*$ is compactly embedded in $L^2(X,\mu)$, i.e., for $(u_n)_{n\in \mathbb{N}}\subset H^1_*$ bounded, there exists a subsequence that converges strongly in $L^2(X,\mu)$.
\end{lemma}

\begin{proof}

First, we prove the compactness in the inner of the domain and then extend this result to the whole space using the smoothness of the boundary.

Take $(u_n)_{n\in\mathbb{N}}\subseteq D(\E)$ with 
$$\int\limits_X |u_n|^2\dd\mu+\E(u_n)\leq c<\infty$$ 
and define $X_\varepsilon\coloneqq \{x\in X:\dist(x,\partial X)<\varepsilon\}$ where $\partial X\coloneqq \bigcup\partial M_i$. We prove that this sequence has a converging subsequence in $L^2$.
We choose an open covering of $X\setminus X_\varepsilon$ by balls by doing this for both $M_i$ individually and taking the union. For this we follow~\cite[Lemma 5.2]{paulik2005gluing}. This covering consists of $q$ balls of fixed radius $r<\varepsilon$ such that 
$$q\lesssim \left(\frac{2(\diam(X)+r)}{r}\right)^\nu$$
and each point is contained in at most 
$l\lesssim 2^{4\nu}$ balls where $\nu$ depends on the doubling constants of $\vol_{M_i}$ and is independent of $r$.
In the following, we call this balls $B_i$ where only the balls restricted to the manifolds are meant. They do not need to be metric balls in $X$.

By Banach-Alaoglu there exists a subsequence and $u\in L^2$ such that $u_n\rightharpoonup u$ in $L^2$. Let $\omega_{n,m}\coloneqq u_n-u_m$ be the difference. We aim to show $\omega_{n,m}\to 0$ in $L^2(X\setminus X_\varepsilon)$.

With $(\omega_{n,m})_i$ as the mean value in $B_i$, we can estimate 
\begin{align*}
    \int\limits_{X\setminus X_\varepsilon}\omega_{n,m}^2\dd\mu&\leq \sum\limits_i\int\limits_{B_i}\omega_{n,m}^2\dd\mu\\
    &\lesssim \sum\limits_i\int\limits_{B_i}|\omega_{n,m}-(\omega_{n,m})_i|^2\dd\mu+ \sum\limits_i\int\limits_{B_i}|(\omega_{n,m})_i|^2\dd\mu\\
    &=:\text{I}+\text{II}.
\end{align*}
The first part I can be estimated by the Poincar\'e inequality on balls on $M_i$, see Definition~\ref{Def:2PI} and Subsection~\ref{SS:A2Prop}:
\begin{align*}
    \text{I}=\sum\limits_i\int\limits_{B_i}|\omega_{n,m}-(\omega_{n,m})_i|^2\dd\mu&\leq r^2\sum\limits_i\int\limits_{B_i}|\nabla \omega_{n,m}|^2\dd\mu\\
    &\leq lr^2\int\limits_X|\nabla \omega_{n,m}|^2\dd\mu\\
    &\lesssim lr^2.
\end{align*}

The second part II directly yields (with a possibly larger $\nu$ that only depends on the doubling constant of $\vol_{M_i}$):
\begin{align*}
    \text{II}&=\sum\limits_i\frac{1}{\mu(B_i)}\left(\int\limits_{B_i}\omega_{n,m}\dd\mu\right)^2\\
    &\leq q\max\limits_i\frac{1}{\mu(B_i)}\left(\int\limits_{B_i}\omega_{n,m}\dd\mu\right)^2\\
    &\lesssim q\frac{1}{r^\nu\mu(X)}\max\limits_i\left(\int\limits_{B_i}\omega_{n,m}\dd\mu\right)^2.
\end{align*}
Here, we used the doubling condition on $M_i$ to compare $\mu(B_i)$ to $\mu(X)$.

Since $\omega_{n,m}\rightharpoonup 0$ in $L^2$ as $n,m\to\infty$, we have $\int\limits_{B_i}\omega_{n,m}\dd\mu\to 0$.
Now, fix $\delta>0$ and plug in $r=\sqrt{\frac{\delta}{l}}$:
\begin{align*}
    \int\limits_{X\setminus X_\varepsilon}\omega_{n,m}^2\dd\mu\lesssim \delta+\frac{ql^\frac{\nu}{2}}{\delta^\frac{\nu}{2}\mu(X)}\max\limits_i\left(\int\limits_{B_i}\omega_{n,m}\dd\mu\right)^2<2\delta.
\end{align*}
This holds for $n,m$ large enough since $q$ is bounded independently of $n,m$. Thus, we can find $n_*,m_*$ dependent on $\delta$ such that for $n>n_*, m>m_*$ this inequality holds. We conclude that $\omega_{n,m}$ and thus $u_n$ converges to $0$ strongly in $L^2$.

To approximate $X$, we need that the derivative cannot concentrate on the boundary.
We define a quantity that measures the mass which can concentrate on the boundary
\begin{align*}
    \Gamma_X(\varepsilon)&\coloneqq \sup\limits_{\norm{u}_{H^1_*(X)}=1}\int\limits_{X_\varepsilon}|u|^2\dd\mu, \\
    \Gamma_X(0)&\coloneqq \lim\limits_{\varepsilon\searrow 0}\Gamma_X(\varepsilon).
\end{align*}
Since the manifolds are smooth with smooth boundary and bounded weights in a neighborhood of the boundary, we have $\Gamma_X(0)=0$, see~\cite[Chapter 5]{paulik2005gluing}.

Take $u_n$ a sequence in $L^2(X)$ such that $\norm{u_n}_{H^1_*(X)}\leq 1\;\forall n$.
Then, for a sequence $\delta_k>0$, we find $\varepsilon_k$ such that $\Gamma_X(\varepsilon_k)<\delta_k$ and a subsequence $u_n$ such that
\begin{align*}
    u_n&\rightharpoonup u \text{ in }L^2(X),\\
    \forall n\geq n_0(k)&: \norm{u-u_n}_{L^2(X\setminus X_{\varepsilon_k})}^2\leq \delta_k.
\end{align*}
The sequence is constructed via a diagonal sequence. First, take a subsequence $u_n\rightharpoonup u$ in $L^2(X)$. Afterwards, for $\varepsilon_k$, we can find refining subsequences such that $u_n\to u$ in $L^2(X\setminus X_{\varepsilon_k})$ and for $n\geq n_0(k)$: 
$$\norm{u-u_n}_{L^2(X\setminus X_{\varepsilon_k})}^2\leq \delta_k.$$
The existence of the whole sequence is given as a diagonal sequence.

Next, we can compute the norm to show strong convergence in $L^2(X)$:
\begin{align*}
    \limsup_{n\to 0}\norm{u_n-u}_{L^2(X)}^2&=\limsup_{n\to 0}\norm{u_n-u}_{L^2(X\setminus X_{\varepsilon_k})}^2+\limsup_{n\to 0}\norm{u_n-u}_{L^2(X_{\varepsilon_k})}^2\\
    &\lesssim \delta_k+\limsup_{n\to 0}\norm{u_n}_{L^2(X_{\varepsilon_k})}^2+\norm{u}_{L^2(X_{\varepsilon_k})}^2\\
    &\lesssim \delta_k+\limsup_{n\to 0}\norm{u_n}_{L^2(X_{\varepsilon_k})}^2\\
    &\lesssim \delta_k+\delta_k \norm{u_n}_{H^1_*(X)}^2\lesssim \delta_k.
\end{align*}
We used the lower semi-continuity of the $L^2$-norm with respect to the weak $L^2$-convergence. With $\delta_k\to 0$, this implies the compact embedding.    
\end{proof}

Therefore, the image of the resolvent $R_{\A}$ is compactly embedded into $L^2$ which yields that $R_{\A}$ is a compact resolvent.

Now, we can combine these calculations and conclude that only constants lie in the kernel of $\A$.\\
Given the above assumptions, $H^1_*$ has a continuity or trace condition, see Remark~\ref{Rem:Trace}. Thus, assume $f\in \text{Ker}(\A)$. Then, $$0=-(\A f,f)=\E(f).$$
On each manifold, we fix charts and take partition of unities on the charts which by the transformation to $\R^{n_i}$ yield 
$$\norm{\nabla f}_{L^2(\mu)}=0.$$
Since $\omega_i>0$, we conclude that $f$ is constant a.e.\ on each chart. By approximation and since the manifolds are connected, $f$ is constant on each manifold. Now, we can infer the value of $f$ on $L$ q.e.\ by the mean value on each manifold and it has to be the same on both sides (since the capacities are equivalent), which shows that $f$ has to be constant in $X$, that is, the kernel of $\A$ consists precisely of constant functions.

Together with the facts that $\A$ is a selfadjoint, nonpositive operator and has a compact resolvent by operator-theory, it follows that the spectrum of $\A$ is discrete with no finite accumulation point and $\A$ has a spectral gap and an orthonormal basis of eigenfunctions, cf.~\cite[Proposition 8.8]{taylor2010partial}. This is equivalent to satisfying a global Poincar\'e type inequality on the orthogonal complement of the constants $\langle 1\rangle^\perp$. To see this, we can write down the condition on the orthogonal complement of the kernel. There exists $\lambda> 0$ such that 
$$\lambda \norm{f-f_X}_{L^2(X,\mu)}\leq -(\A f, f)=\mathcal{E}(f).$$
Next, we apply the spectral theorem which represents $\A$ as a multiplication operator $a\leq 0$ on $L^2(X)$ with $a\leq -\lambda$ on the orthogonal complement of the constants. For $f$ with $\int f\dd\mu=0$, this yields $$\norm{S_tf}_{L^2(\mu)}\leq e^{-\lambda t}\norm{f_0}_{L^2(\mu)}\to 0$$
as $t\to \infty$. Therefore, solutions converge exponentially fast to the constants that are their mean values.

This, in turn, proves the irreducibility or ergodicity of $S_t$ following the Definitions~\ref{def:irreducible} and \ref{def:ergodic}:

\begin{proof}[Proof of Theorem~\ref{mainThm}]
    First, we prove the ergodicity. If $E$ satisfies $T_t^{-1}(E)= E$ then we also have 
    $$T_t(E)=T_t(T_t^{-1}(E))\subseteq E.$$
    By monotonicity, we know that if $A\subseteq B$ then
    $$T_t(A)\subseteq T_t(B).$$
    This holds as $T_t$ is monotone since $S_t$ is linear and monotone (positive) as shown now. By \cite[Theorem 1.4.1]{fukushima2010dirichlet}, $S_t$ is Markovian, as it is the semigroup associated with a Dirichlet form, that is, if $0\leq f\leq 1$ then $0\leq S_hf\leq 1$. Thus, by linearity of $S_t$, it is positivity preserving on bounded function and therefore monotone.
    
    By the upper bound ($S_t\chi_A\leq 1$ a.e.), $S_s\chi_A\leq T_s(A)$ and thus
    $$T_{t+s}(A)\subseteq T_t\circ T_s(A).$$
    If $\mu(E)>0$, we have $T_t(E)\to X$ as $t\nearrow \infty$ by the observation that all functions converge to their mean as a constant function under the flow. If $T_t(E)\subseteq E$ for $t>0$ then 
    $$T_{2t}(E)\subseteq T_t\circ T_t(E)\subseteq T_t(E)\subseteq E.$$ 
    Inductively this implies that a.e.\ point of $X$ is contained in $E$ and thus $\mu(E^c)=0$.

    To prove irreducibility, we can plug in $f\equiv 1$ and get $S_h(\chi_E)=\chi_E$ $\mu$-a.e.\ which implies $T_h(E)=E$ a.e.
\end{proof}

\section{Generalizations}
\label{Generalization}

In this section, we show how the proof and techniques of Section~\ref{manifoldSetting} can be generalized to the case of more than two connected manifolds. Moreover, we give an idea on how to incorporate more general weights.

\begin{theorem}[Heat flow on the glued manifold II]
\label{Thm:GenMainThmHeatGlued}
    Let $(X,d,\mu)$ be a metric space with $X=\bigcup\limits_{i=1}^m M_i$ being a union of smooth, bounded, compact, connected, Riemannian manifolds whose intersection $L_{ij}\coloneqq M_i\cap M_j$ are disjoint, transversal, nonempty and lie in the inner of both manifolds. Equip $X$ with the glued distance $d$ and assume that $X$ is path-wise connected. Furthermore, let $$\mu\coloneqq \sum\limits_{i=1}^m \omega_i\cdot\text{vol}_{M_i}$$
    fulfill $\mu(L_{ij})=0$ and be locally $N_{ij}$-doubling.
    
    Moreover, assume nonnegative weights $\omega_i\in A_2(M_i),\omega_i>0$ such that they are bounded in a neighborhood of $\partial M_i$ and on each intersection $L\in\{L_{ij}\}$ the following holds for one of both manifolds:
    \begin{itemize}
        \item $\mu_i$ satisfies the $L_{ij}$-Muckenhoupt type condition, see Definition~\ref{Def:LMucken},
        \item $\int\limits_0^R \frac{N_{ij}(\rho)\rho}{\mu_i(L_\rho)}\dd \rho,\int\limits_0^R \frac{\rho}{\mu_i(L_\rho)}\dd \rho<\infty$.
    \end{itemize}

    Then, the canonical heat flow is ergodic and irreducible.
\end{theorem}
\begin{remark*}
    In Example~\ref{Ex:generalWeights}, we give a class of measures that satisfy the $L$-Mucken\-houpt condition. The other conditions are required for the capacity to be positive, and thus for the connectedness of the manifolds. These can be read intuitively as the fact that the effective codimension of the intersection must be strictly between $0$ and $2$ and that the jump in the effective dimension cannot be too large.
\end{remark*}

As in the case of two manifolds, we still rely on the equivalence of the capacities to get a trace in our Sobolev space. Thus, the effective dimension (taking into account the weight) of the manifolds at their intersection cannot vary too much. In terms of heuristically speaking, the difference of $k$ and the effective dimension minus one must be less than one. We first prove the case of two manifolds with general weights, as in Theorem~\ref{Thm:GenMainThmHeatGlued}, intersecting, and then discuss that all the proofs carry over to the case of disjoint intersections.

We begin with lower bounds. For general weights $\omega_1, \omega_2\in A_2$ on $M_i$ with $\text{dim}(M_i)=n_i$, we can follow the same calculations as above to obtain lower bounds on the capacities:

\begin{align*}
    {\Capp}_{2,\mu_1}(L, L_R)&\gtrsim \left(N(R)R^{-2(n_2-k-1)}\frac{\LL^{n_2}}{\omega_2}(L_R)\right.\\
    &\quad\left.+2(n_1-k-1)\int\limits_0^R N(\rho)\rho^{-(2n_2-2k-1)}\frac{\LL^{n_2}}{\omega_2}(L_\rho)\dd \rho\right)^{-1},\\
    {\Capp}_{2,\mu_2}(L, L_R)&\gtrsim \left(R^{-2(n_2-k-1)}\frac{\LL^{n_2}}{\omega_2}(L_R)\right.\\
    &\quad\left.+2(n_2-k-1)\int\limits_0^R \rho^{-(2n_2-2k-1)}\frac{\LL^{n_2}}{\omega_2}(L_\rho)\dd \rho\right)^{-1}.
\end{align*}

Additionally, if $\omega_2$ satisfies the $L$-Muckenhoupt type condition on the intersection, see Definition~\ref{Def:LMucken}, we can simplify these expressions:

\begin{align*}
    {\Capp}_{2,\mu_1}(L, L_R)&\gtrsim \left(N(R)R^2\mu_2(L_R)^{-1}+2(n_1-k-1)\int\limits_0^R N(\rho)\rho\mu_2(L_\rho)^{-1}\dd \rho\right)^{-1},\\
    {\Capp}_{2,\mu_2}(L, L_R)&\gtrsim \left(R^2\mu_2(L_R)^{-1}+2(n_2-k-1)\int\limits_0^R \rho\mu_2(L_\rho)^{-1}\dd \rho\right)^{-1}.
\end{align*}

Moreover, we can prove an upper capacity bound, which shows that our lower bounds are optimal up to a constant. This upper bound on the capacity can be obtained through an explicit competitor. First, consider $L_r$ in $L_{2r}$ and the linear competitor $u(x)=\left(1-\frac{\dist(x,L_r)}{r}\right)\chi_{L_{2r}}\in \Lip\cap C_c(L_{2r})$:
\begin{align*}
    {\Capp}_{2,\mu_2}(L_r, L_{2r})&\lesssim \int\limits_{L_{2r}\setminus L_r}r^{-2}\dd\mu_2\\
    &\leq \int\limits_{L_{2r}}r^{-2}\dd\mu_2\\
    &\lesssim r^{2n_2-2k-2}\left(\int\limits_{L_{2r}}\frac{\dd\LL^{n_2}}{\omega_2}\right)^{-1}\\
    &\leq r^{2n_2-2k-2}\left(\int\limits_{L_{2r}\setminus L_r}\frac{\dd\LL^{n_2}}{\omega_2}\right)^{-1}\\
    &\leq \left(\int\limits_{L_{2r}\setminus L_r}\frac{1}{\dist(y,L)^{2(n_2-k-1)}}\frac{\dd\LL^{n_2}}{\omega_i}\right)^{-1}.
\end{align*}
This can be used iteratively to compute the capacity. With $l\in\mathbb{N}$ such that $2^{l-1}r\leq R\leq 2^lr$, we get:
\begin{align*}
    {\Capp}_{2,\mu_2}(L_r, L_R)&\lesssim {\Capp}_{2,\mu_2}(L_r, L_{2^lr})\\
    &\leq \left(\sum\limits_{j=0}^{l-1} {\Capp}_{2,\mu_2}(L_{2^jr}, L_{2^{j+1}r})^{-1}\right)^{-1}\\
    &\leq \left(\int\limits_{L_{2^lr}\setminus L_r}\frac{1}{\dist(y,L)^{2(n_2-k-1)}}\frac{\dd\LL^{n_2}}{\omega_2}\right)^{-1}\\
    &\leq \left(\int\limits_{L_{R}\setminus L_r}\frac{1}{\dist(y,L)^{2(n_2-k-1)}}\frac{\dd\LL^{n_2}}{\omega_2}\right)^{-1}.
\end{align*}
The proof of the first inequality is analogous to~\cite[Lemma 2.16]{heinonen2018nonlinear}.
With the simplification of Section~\ref{manifoldSetting},
\begin{align*}
    {\Capp}_{2,\mu_2}(L, L_R)&\lesssim \left( R^{-2(n_2-k-1)}\frac{\LL^{n_2}}{\omega_2}(L_R)\right.\\
    &\quad\quad\left.+2(n_2-k-1)\int\limits_0^R \rho^{-(2n_2-2k-1)}\frac{\LL^{n_2}}{\omega_2}(L_\rho)\dd \rho\right)^{-1}.
\end{align*}

With $\omega_2$ satisfying the Muckenhoupt type condition, this yields
\begin{align*}
    {\Capp}_{2,\mu_2}(L, L_R)&\lesssim \left(R^2\mu_2(L_R)^{-1}+2(n_2-k-1)\int\limits_0^R \rho\mu_2(L_\rho)^{-1}\dd \rho\right)^{-1}.
\end{align*}

Now, we can conclude that the capacities are positive under the assumptions of Theorem~\ref{Thm:GenMainThmHeatGlued}. By assumption and since $R>0$ is fixed, we have $$N(R)R^2\mu_2(L_R)^{-1}, R^2\mu_2(L_R)^{-1}<\infty.$$
Additionally, we assumed $\int\limits_0^R \frac{N(\rho)\rho}{\mu_2(L_\rho)}\dd \rho,\int\limits_0^R \frac{\rho}{\mu_2(L_\rho)}\dd \rho<\infty$
and thus by our lower bound, we get 
$${\Capp}_{2,\mu_1}(L, L_R),{\Capp}_{2,\mu_2}(L, L_R)>0.$$

\begin{example}
\label{Ex:generalWeights}
    The following example demonstrates the effective dimension condition when the weight is locally a rational function around the intersection. Additionally, it shows that the $L$-Muckenhoupt type condition is reasonable.
    
    We consider a submanifold $L$ in a manifold $M$ with $n\coloneqq \dim M$ and $k\coloneqq \dim L\neq n-1$ with weight $\omega$ which locally around $L$ has the form $$\omega(x)\coloneqq \frac{1}{\dist(x,L)^\alpha}.$$ 
    Then locally, 
    \begin{align*}
        \mu_i(L_r)&\simeq\begin{cases}
            \frac{r^{n-k-\alpha}}{n-k-\alpha} &\text{if }\alpha<n-k,\\
            \infty &\text{else,}
        \end{cases}\\
        \frac{\LL^{n}}{\omega}(L_r)=\int\limits_{L_r}\frac{1}{\omega}\dd x&\simeq\begin{cases}
            \frac{r^{n-k+\alpha}}{n-k+\alpha} &\text{if }\alpha>-(n-k),\\
            \infty &\text{else.}
        \end{cases}
    \end{align*}
    Thus, as $\mu$ needs to be locally finite and the capacity needs to be finite, we require $\alpha\in (-(n-k),n-k)$. In this case, the $L$-Muckenhoupt type condition is satisfied and the requirement for the capacities becomes
    \begin{align*}
        \int\limits_0^R \rho\mu_i(L_\rho)^{-1}\dd \rho&=\int\limits_0^R \rho^{1-n+k+\alpha}\dd \rho<\infty,
    \end{align*}
    i.e., $\alpha>n-k-2$, which coincides with our intuition that the difference of $k$ and effective dimension minus one, $n-\alpha-1$, has to be less than one.
    
    If $L$ is the intersection with another manifold and the whole measure $\mu$ satisfies the $N$-doubling condition, the condition for the other manifold is
    \begin{align*}
        \int\limits_0^R N(\rho)\rho^{1-n+k+\alpha}\dd \rho<\infty.
    \end{align*}
\end{example}

Next, we describe the iteration procedure for more than two manifolds. We define the Dirichlet form analogously to before
$$\E(f)\coloneqq \int\limits_X|\nabla f|^2\dd\mu\coloneqq \sum\limits_{i=1}^m \int\limits_{M_i}|\nabla^{M_i} f|^2\dd\mu_i.$$

In this setting, we can use the same proofs as before since the calculations for the capacities are done locally around the intersections and the other proofs are done analogously for a finite number of manifolds. Therefore, we still can use that $(X,d)$ is a metric doubling space to prove the compact embedding. We need to use the connectedness to conclude that the kernel again consists only of the constants. Then, the capacity conditions, the ergodicity and irreducibility follow analogously. 

Moreover, we can treat other boundary conditions and manifolds in a similar fashion. If we additionally have the existence of a heat kernel with suitable upper and lower exponential bounds, we can use the localization result in metric measure spaces, see~\cite[Theorem 3.22]{post2018locality}, and the semigroup property to extend our results to more general settings. This is due to the fact an everywhere positive heat kernel implies ergodicity and irreducibility. In this setting, only the local geometry matters and thus, different boundary conditions can be treated as well.

\section{Outlook}
\label{Outlook}

In the following, we give the notation and definitions needed to define a non-local perimeter functional, the heat excess $E_h$. An interesting question is whether this functional converges to the perimeter as $h$ goes to zero. The correct notion of convergence for this is the one of $\Gamma$-convergence, cf.~Definition~\ref{Def:Gamma-Convergence}. This convergence is proven in the context of Euclidean space, see~\cite{EsedogluOttoThresh}, and smooth manifolds, see~\cite{jonatim}, but it is unknown in our setting or for general metric measure spaces. In this context, it is known that the limit leads to an equivalent characterization of $\BV$ under assumptions on the measure; cf.~\cite{ruiz2021gagliardo, ruiz2020heat, alonso2018bv, alonsoruiz2020besov, alonso2021besov, Marola_2016}.

We begin with the definition of $\Gamma$-convergence.

\begin{definition}[$\Gamma$-convergence]
    \label{Def:Gamma-Convergence}
    Let $X$ be a topological space and $(F_n)_{n\in\mathbb{N}}$ be a sequence of functionals $F_n:X\to\mathbb{R}\cup\{\infty\}$.
    Then, we say that $(F_n)_{n\in\mathbb{N}}$ $\Gamma$-converges to a functional $F:X\to\mathbb{R}\cup\{\infty\}$, written $F_n\overset{\Gamma}{\to}F$ if
    \begin{itemize}
        \item for all $x_n\to x, x_n,x\in X$ we have the liminf inequality
        $$\liminf\limits_n F_n(x_n)\geq F(x),$$
        \item for all $x\in X$ there exists a recovery sequence $x_n\to x, x_n\in X$ such that 
        $$\lim\limits_n F_n(x_n)= F(x).$$
    \end{itemize}
\end{definition}

\begin{remark*}
    The second property is often written as the equivalent formulation of a limsup inequality $$\limsup\limits_n F_n(x_n)\leq F(x).$$
\end{remark*}

We define the perimeter functional for a function $f\in L^1(X,\mu)$ in accordance with~\cite{Marola_2016} as:
$$\norm{\nabla f}(U)\coloneqq \liminf\left\{\norm{\nabla f_n}(U): f_n\to f \text{ in } L^1(X), f_n\in \Lip(X) \right\}.$$

Additionally, we define the heat excess at time $h>0$ via
$$E_h(f)\coloneqq \int\limits_X S_h|f(x)-f(\cdot)|(x)\dd\mu(x).$$
If the heat semigroup $S_h$ possesses a heat kernel $p_h(x,y)$, we can rewrite it as
$$E_h(f)\coloneqq \int\limits_X\int\limits_X p_h(x,y)|f(x)-f(y)|\dd\mu(y)\dd\mu(x).$$ 
If $f=\chi_U$ is a characteristic function, it becomes 
$$E_h(f)=\int\limits_X\chi_{U^c}S_h\chi_U\dd\mu(x).$$ 
The heat excess can also be seen as a Korevaar-Schoen type energy. Thus, the convergence is related to the Bourgain-Brezis-Mironescu formula, cf.~\cite{gorny2020bourgainbrezismironescu}.

An interesting question is whether the following conjecture holds in general or even if it holds in the setting of glued manifolds with $A_1$ weights and doubling measure.
\begin{conjecture*}[Perimeter functional]
   Let $(X,d,\mu)$ be a complete metric space with $(\E, D(\E))$ a strictly local regular Dirichlet form. Moreover, assume that $\mu$ satisfies an $1$-Poincar\'e inequality and is doubling. Then, for $f\in BV(X)$:
   $$E_h(f)\overset{L^1-\Gamma}{\longrightarrow} \norm{\nabla f}(X).$$
\end{conjecture*}

Given this convergence, we can construct a classifier by considering a minimizing movement scheme similar to the thresholding scheme (MBO-scheme) to receive a fast iterative algorithm to classify high-dimensional data. 
Schemes like these are analyzed in \cite{medianfilter}.

In proving the $\Gamma$-convergence of the heat excess $E_h(f)\overset{\Gamma}{\to}\norm{\nabla f}$ in $L^1$, it is necessary for $
\mu$ to satisfy a $1$-Poincar\'e inequality (or similar condition on the measure $\mu$).
\begin{definition}[$p$-Poincar\'e inequality]
\label{Def:pPoincar\'e}
    We say that $(X,d,\mu)$ satisfies a $(q,p)$-Poincar\'e inequality if there exists a constant $c>0$ and a dilation factor $\lambda\geq 1$ such that for all integrable $f$ and for each ball $B$ of radius $r$:
    \begin{align*}
        \left(\fint_B |f-f_B|^q\dd\mu\right)^\frac{1}{q}&\leq cr\left(\fint_{\lambda B} |\nabla f|^p\dd\mu\right)^\frac{1}{p}.
    \end{align*}
    If $p=q$, we call this condition also a $p$-Poincar\'e inequality.
\end{definition}
\begin{remark*}
    If $X$ satisfies a $(q,p)$-Poincar\'e inequality, then we also have a $(q',p')$-Poincar\'e inequality for $1\leq q'\leq q, p\leq p'$. Moreover, we can also recover a $(p,p)$-Poincar\'e inequality in doubling spaces, cf.~\cite[Theorem 5.1]{SobolevMetPoincare}.
\end{remark*}

In the setting of glued manifolds, for an inequality $1$-Poincar\'e to hold throughout the space, it is necessary that the components satisfy the inequality and a $1$-capacity condition must be satisfied at the intersections; cf.~\cite{BowTies}.

\begin{definition}[$p$-Capacities]
\label{Def:p-Capacity}
    The relative $(p,\mu)$-capacity is defined on compact sets and extended to measurable sets in the usual way:
    \begin{align*}
        {\Capp}_{p,\mu}(K,\Omega)\coloneqq \inf\left\{\int\limits_\Omega |\nabla u|^p\dd\mu:u\in \Lip\cap C_c(\Omega), u\geq \chi_K\right\}.
    \end{align*}
\end{definition}

Moreover, with $1$-Muckenhoupt weights, the $1$-Poincar\'e inequality is satisfied in the individual manifolds.

\begin{definition}[$p$-Muckenhoupt weights]
    \label{Def:p-Muckenhoupt}
    We say $\omega$ is in the class $A_p(M)$ of $p$-Muckenhoupt weights on a smooth compact manifold $M$ for $p>1$ if for all balls $B$
    $$\fint\limits_B \omega\dd\vol_{M}\cdot \left(\fint\limits_B \frac{1}{\omega^{\frac{1}{p-1}}}\dd\vol_{M}\right)^{p-1}\lesssim 1$$
    and $\omega\in A_1(M)$ if
    $$\fint\limits_B \omega\dd\vol_{M}\simeq \essinf_B \omega.$$
\end{definition}

\begin{remark*}
    The spaces of Muckenhoupt weights are nested, i.e., $A_1\subseteq A_p\subseteq A_q$ for $1<p<q$, cf.~\cite[Proposition A.1]{AuscherMartellII}.
\end{remark*}

We can now ask the question of the $\Gamma$- or Mosco-convergence of $E_h$ to $\norm{\nabla f}$. To this end, different inequalities as well as bounds for the heat kernel are interesting. We note that for this one has to consider the $A_1$ weights and $1$-Poincar\'e inequalities instead of the $L^2$ variants. Moreover, to carry over these inequalities from the components to the whole space, one needs more refined geometric assumptions on the intersection related to ${\Capp}_{1,\mu}$. 

\section{Related work}
\label{background}

It is not feasible to give an extensive overview of the heat equation, Brownian motion, heat kernel estimates, and flows on manifolds and metric measure spaces. Thus, we list some of the directly connected more recent works in these areas that are relevant to this paper and our setting.

In this work, we consider the setting of glued manifolds with weights. This setting is motivated by manifold learning, and examples of similar approaches were studied before in the works we list next. There are different approaches that one could follow to connect manifolds. In \cite{chen2019brownian, Takumud1d2} the space $\mathbb{R}^2\times \mathbb{R}_+$ with a collapsed center is analyzed and an expansion of the heat kernel and Brownian motion is computed. Similarly, \cite{Shuwen3D, Shuwen3D2} study the so-called distorted Brownian motion in $\mathbb{R}^3\times \mathbb{R}_+$ where the connection is done via a density that fulfills our conditions. Using the same setting, in \cite{Takumud1d2}, Ooi extends the results of \cite{chen2019brownian} to the case of two arbitrary $\mathbb{R}^{n_1},\mathbb{R}^{n_2}$ meeting at one collapsed point with explicit heat kernel bounds. The works \cite{BowTies, christensen2023capacities} study bow ties as special structures and how Poincar\'e inequalities on them can be obtained from capacity conditions. In \cite{paulik2005gluing}, glued spaces with constant dimension are analyzed. A major difference is that only manifolds of the same dimension are considered and they do not possess a density.

In a more recent work than this paper, namely \cite{BungertSlepcevDirichletEnergies}, the authors consider Dirichlet energies on intersection manifolds of potentially different dimensions with bounded densities. They derive a formulation that respects the intrinsic dimensions without intrinsic knowledge of them. In accordance with this paper, they obtain the result that the energies separate if the jump in dimension is too high. Furthermore, the energy connects the spaces via a flow condition of the derivative in the case that the manifolds are of the same dimension and intersect in a manifold that has codimension one in them.

The theory of Muckenhoupt weights and related study of them can be found in various papers and books. The main sources for this work were \cite[Chapter 2 and 15]{heinonen2018nonlinear}, \cite{kilpelainen1994weighted} and \cite[Chapter 5]{stein1993harmonic}, where they are introduced and studied together with the related capacities. For the setting of metric measure spaces, one can see \cite{kurki2022extension} or, in the context of non-doubling measures, \cite{Yasuo}.

In their series of four papers~\cite{AuscherMartellI, AuscherMartellII, AuscherMartellIII, AuscherMartellIV}, the authors investigate, among other things, Muckenhoupt weights on manifolds, where the appendix in \cite{AuscherMartellII} is especially noteworthy.

Related to our construction is the theory of Dirichlet forms and Markov processes in \cite{fukushima2010dirichlet} and linear operators in \cite{kato2013perturbation} as well as capacities, see~\cite{kilpelainen1994weighted, kinnunen1996sobolev}.

The theory of heat flows and kernels with their estimates is extensive and is contained in many fields. In \cite{grigor2009heat}, Grigor'yan and co-authors provide a condition for the reverse doubling condition on manifolds and prove connections between measure doubling conditions and heat kernel estimates. In \cite{Grigoryan}, he provides a general theory of the upper and lower bounds of heat kernels and how to achieve these in a general setting.
Other estimates of heat kernels by Grigor'yan are proven in \cite{grigor2006heat, grigoryan2009heat, grigor2022off, grigor2014heat, grigor2012two}.

For the heat kernel, in \cite{sturm1995analysis} and \cite{sturm1994analysis}, Sturm shows bounds under the doubling assumption on the measure.
For the bounds on heat kernels, \cite{carlen1987upper} provides theorems that imply an upper bound given a Sobolev inequality of Nash type. This Nash inequality in manifolds can be proven using \cite{kumura2001nash} where a proof is given for weighted manifolds which fit our setting. Other Nash inequalities can be found in \cite{bakry2010weighted, kigami2004local}.

In \cite{kumagai2005construction}, a heat flow and Brownian motion are constructed from Dirichlet forms which arise as limits of discrete forms.
For the expansion of the heat kernel on manifolds with smooth weights, a parametrix construction is needed, which is done in \cite{rosenberg1997laplacian}. Even in a more general setting, the heat kernel is localized, which can be seen in the locality theorems and a comparison result in \cite{post2018locality}.
In \cite{bifulco2023lipschitz}, the Lipschitz continuity of heat kernels in metric measure spaces is investigated.
More bounds on the heat kernel can be found in \cite{BoutayebHeat, coulhonoff}.

Furthermore, in \cite{GrigoryanSaloffCosteConnected}, heat kernel bounds on connected manifolds are studied. For the setting, the authors consider same dimensional manifolds that potentially have different scalings in the long-time and large-scale behavior. These are connected through a common set so that a manifold with ends is created for whose heat kernel bounds are derived.

In Section~\ref{Definitions}, we work with Sobolev spaces that represent the domain of the Dirichlet form and the Laplace-Beltrami operator.
In \cite{hajlasz2003sobolev}, Haj{\l}asz defines different notions of Sobolev spaces, and in \cite{SobolevMetPoincare}, he proves, for instance, Sobolev embeddings of them. A short summary of Haj{\l}asz Sobolev-spaces is given in \cite{hajlasz1996sobolev}.
Generalizations of embeddings of Sobolev spaces can be found in \cite{gurka1988continuous, kalamajska1999compactness}. Here, the former analyzes the setting of embeddings in weighted Sobolev spaces, whereas the latter gives a generalization of the Rellich-Kondrachov embedding for the Haj{\l}asz Sobolev-space. This proof uses a similar criterion to our $N$-doubling condition.

For the classifier and definition perimeter functional, we need to define functions of bounded variations. These possess numerous different characterizations in metric spaces. 
The works \cite{buffa2021bv} and \cite{miranda2003functions} give an overview and characterization of different definitions.
Moreover, in \cite{ambrosio2016bmo} a BMO characterization of sets of finite perimeter is given.

The convergence of the heat excess has already been studied in different cases. The most relevant ones for this setting are \cite{ruiz2021gagliardo, ruiz2020heat, alonso2018bv, alonsoruiz2020besov, alonso2021besov, EsedogluOttoThresh, jonatim, Marola_2016}.
First, in \cite{EsedogluOttoThresh} the convergence in Euclidean space for the multiphase mean curvature flow setting is shown.
In \cite{jonatim}, the $\Gamma$-convergence was proven in the setting of smooth manifolds.
\cite{ruiz2021gagliardo} as well as \cite{alonso2018bv}, \cite{alonsoruiz2020besov} and \cite{alonso2021besov} use a weak Bakry-\'Emery condition in the setting of metric measure spaces with a $2$-Poincar\'e inequality to establish the equivalence of the limit with the BV seminorm.
In comparison, \cite{Marola_2016} and \cite{ruiz2020heat} prove similar results using a $1$-Poincar\'e inequality. The latter is in the setting of smooth manifolds where pointwise convergence is shown.

Similarly to this question, one can ask for a Bourgain-Brezis-Mironescu formula in metric measure spaces, i.e., whether a nonlocal kernel perimeter converges to the perimeter.
This was investigated in \cite{di2019new, gorny2020bourgainbrezismironescu, han2021asymptotic, kreuml2019fractional, lahti2023bv, lahti2024characterization}.

\section{Acknowledgement}

I would like to thank my supervisor Prof.\ Felix Otto for many fruitful discussions and help in this work.
Many thanks also to Prof.\ Tim Laux for the discussions we had.

\frenchspacing
\bibliographystyle{abbrv}
\bibliography{References.bib} 

\begin{thebibliography}{10}

\bibitem{ruiz2021gagliardo}
P.~Alonso~Ruiz and F.~Baudoin.
\newblock {Gagliardo-Nirenberg}, {Trudinger}-{Moser} and {Morrey} inequalities on {Dirichlet} spaces.
\newblock {\em J. Math. Anal. Appl.}, 497(2):27, 2021.
\newblock Id/No 124899.

\bibitem{ruiz2020heat}
P.~Alonso~Ruiz and F.~Baudoin.
\newblock Yet another heat semigroup characterization of {BV} functions on {Riemannian} manifolds.
\newblock {\em Ann. Fac. Sci. Toulouse, Math. (6)}, 32(3):577--606, 2023.

\bibitem{alonso2018bv}
P.~Alonso-Ruiz, F.~Baudoin, L.~Chen, L.~Rogers, N.~Shanmugalingam, and A.~Teplyaev.
\newblock {BV} functions and {Besov} spaces associated with {Dirichlet} spaces.
\newblock Preprint, {arXiv}:1806.03428 [math.{FA}] (2018), 2018.

\bibitem{alonsoruiz2020besov}
P.~Alonso-Ruiz, F.~Baudoin, L.~Chen, L.~Rogers, N.~Shanmugalingam, and A.~Teplyaev.
\newblock Besov class via heat semigroup on {Dirichlet} spaces. {II}: {BV} functions and {Gaussian} heat kernel estimates.
\newblock {\em Calc. Var. Partial Differ. Equ.}, 59(3):32, 2020.
\newblock Id/No 103.

\bibitem{alonso2021besov}
P.~Alonso-Ruiz, F.~Baudoin, L.~Chen, L.~Rogers, N.~Shanmugalingam, and A.~Teplyaev.
\newblock Besov class via heat semigroup on {Dirichlet} spaces. {III}: {BV} functions and sub-{Gaussian} heat kernel estimates.
\newblock {\em Calc. Var. Partial Differ. Equ.}, 60(5):38, 2021.
\newblock Id/No 170.

\bibitem{ambrosio2016bmo}
L.~Ambrosio, J.~Bourgain, H.~Brezis, and A.~Figalli.
\newblock {BMO}-type norms related to the perimeter of sets.
\newblock {\em Commun. Pure Appl. Math.}, 69(6):1062--1086, 2016.

\bibitem{AuscherMartellI}
P.~Auscher and J.~M. Martell.
\newblock Weighted norm inequalities, off-diagonal estimates and elliptic operators. {I}: {General} operator theory and weights.
\newblock {\em Adv. Math.}, 212(1):225--276, 2007.

\bibitem{AuscherMartellII}
P.~Auscher and J.~M. Martell.
\newblock Weighted norm inequalities, off-diagonal estimates and elliptic operators. {II}: {Off}-diagonal estimates on spaces of homogeneous type.
\newblock {\em J. Evol. Equ.}, 7(2):265--316, 2007.

\bibitem{AuscherMartellIII}
P.~Auscher and J.~M. Martell.
\newblock Weighted norm inequalities, off-diagonal estimates and elliptic operators. {III}: {Harmonic} analysis of elliptic operators.
\newblock {\em J. Funct. Anal.}, 241(2):703--746, 2007.

\bibitem{AuscherMartellIV}
P.~Auscher and J.~M. Martell.
\newblock Weighted norm inequalities, off-diagonal estimates and elliptic operators. {IV}: {Riesz} transforms on manifolds and weights.
\newblock {\em Math. Z.}, 260(3):527--539, 2008.

\bibitem{bakry2010weighted}
D.~Bakry, F.~Bolley, I.~Gentil, and P.~Maheux.
\newblock Weighted {Nash} inequalities.
\newblock {\em Rev. Mat. Iberoam.}, 28(3):879--906, 2012.

\bibitem{bifulco2023lipschitz}
P.~Bifulco and D.~Mugnolo.
\newblock On the {Lipschitz} continuity of the heat kernel.
\newblock Preprint, {arXiv}:2307.08889 [math.{FA}] (2023), 2023.

\bibitem{BowTies}
A.~Bj{\"o}rn, J.~Bj{\"o}rn, and A.~Christensen.
\newblock Poincar{\'e} inequalities and {{\(A_p\)}} weights on bow-ties.
\newblock {\em J. Math. Anal. Appl.}, 539(1):28, 2024.
\newblock Id/No 128483.

\bibitem{BoutayebHeat}
S.~Boutayeb, T.~Coulhon, and A.~Sikora.
\newblock A new approach to pointwise heat kernel upper bounds on doubling metric measure spaces.
\newblock {\em Adv. Math.}, 270:302--374, 2015.

\bibitem{UnionManifold}
B.~C.~A. Brown, A.~L. Caterini, B.~L. Ross, J.~C. Cresswell, and G.~Loaiza-Ganem.
\newblock Verifying the union of manifolds hypothesis for image data, 2023.

\bibitem{buffa2021bv}
V.~Buffa, G.~E. Comi, and M.~j. Miranda.
\newblock On {BV} functions and essentially bounded divergence-measure fields in metric spaces.
\newblock {\em Rev. Mat. Iberoam.}, 38(3):883--946, 2022.

\bibitem{BungertSlepcevDirichletEnergies}
L.~Bungert and D.~Slep{\v{c}}ev.
\newblock Convergence of graph {Dirichlet} energies and graph {Laplacians} on intersecting manifolds of varying dimensions.
\newblock Preprint, {arXiv}:2509.24458 [math.{AP}] (2025), 2025.

\bibitem{carlen1987upper}
E.~A. Carlen, S.~Kusuoka, and D.~W. Stroock.
\newblock Upper bounds for symmetric {Markov} transition functions.
\newblock {\em Ann. Inst. Henri Poincar{\'e}, Probab. Stat.}, 23:245--287, 1987.

\bibitem{chen2019brownian}
Z.-Q. Chen and S.~Lou.
\newblock Brownian motion on some spaces with varying dimension.
\newblock {\em Ann. Probab.}, 47(1):213--269, 2019.

\bibitem{christensen2023capacities}
A.~Christensen.
\newblock {\em Capacities, Poincar{\'e} inequalities and gluing metric spaces}.
\newblock Linkopings Universitet (Sweden), 2023.

\bibitem{coulhonoff}
T.~Coulhon.
\newblock Off-diagonal heat kernel lower bounds without {Poincar{\'e}}.
\newblock {\em J. Lond. Math. Soc., II. Ser.}, 68(3):795--816, 2003.

\bibitem{di2019new}
S.~Di~Marino and M.~Squassina.
\newblock New characterizations of {Sobolev} metric spaces.
\newblock {\em J. Funct. Anal.}, 276(6):1853--1874, 2019.

\bibitem{EsedogluOttoThresh}
S.~Esedo{\=g}lu and F.~Otto.
\newblock Threshold dynamics for networks with arbitrary surface tensions.
\newblock {\em Commun. Pure Appl. Math.}, 68(5):808--864, 2015.

\bibitem{fukushima2010dirichlet}
M.~Fukushima, Y.~Oshima, and M.~Takeda.
\newblock {\em Dirichlet forms and symmetric {Markov} processes.}, volume~19 of {\em De Gruyter Stud. Math.}
\newblock Berlin: Walter de Gruyter, 2nd revised and extended ed. edition, 2011.

\bibitem{gorny2020bourgainbrezismironescu}
W.~G{\'o}rny.
\newblock {Bourgain-Brezis-Mironescu} approach in metric spaces with {Euclidean} tangents.
\newblock {\em J. Geom. Anal.}, 32(4):22, 2022.
\newblock Id/No 128.

\bibitem{Grigoryan}
A.~Grigor'yan.
\newblock Estimates of heat kernels on {Riemannian} manifolds.
\newblock In {\em Spectral theory and geometry. Proceedings of the ICMS instructional conference, Edinburgh, UK, 30 March--9 April, 1998}, pages 140--225. Cambridge: Cambridge University Press, 1999.

\bibitem{grigor2006heat}
A.~Grigor'yan.
\newblock Heat kernels on weighted manifolds and applications.
\newblock In {\em The ubiquitous heat kernel. AMS special session, Boulder, CO, USA, October 2--4, 2003}, pages 93--191. Providence, RI: American Mathematical Society (AMS), 2006.

\bibitem{grigoryan2009heat}
A.~Grigor'yan.
\newblock {\em Heat kernel and analysis on manifolds}, volume~47 of {\em AMS/IP Stud. Adv. Math.}
\newblock Providence, RI: American Mathematical Society (AMS); Somerville, MA: International Press, 2009.

\bibitem{grigor2022off}
A.~Grigor'yan, E.~Hu, and J.~Hu.
\newblock Off-diagonal lower estimates and {H{\"o}lder} regularity of the heat kernel.
\newblock {\em Asian J. Math.}, 27(5):675--770, 2023.

\bibitem{grigor2009heat}
A.~Grigor'yan, J.~Hu, and K.-S. Lau.
\newblock Heat kernels on metric spaces with doubling measure.
\newblock In {\em Fractal geometry and stochastics IV. Proceedings of the 4th conference, Greifswald, Germany, September 8--12, 2008}, pages 3--44. Basel: Birkh{\"a}user, 2009.

\bibitem{grigor2014heat}
A.~Grigor'yan, J.~Hu, and K.-S. Lau.
\newblock Heat kernels on metric measure spaces.
\newblock In {\em Geometry and analysis of fractals. Based on the international conference on advances of fractals and related topics, Hong Kong, China, December 10--14, 2012}, pages 147--207. Berlin: Springer, 2014.

\bibitem{GrigoryanSaloffCosteConnected}
A.~Grigor'yan and L.~Saloff-Coste.
\newblock Heat kernel on connected sums of {Riemannian} manifolds.
\newblock {\em Math. Res. Lett.}, 6(3-4):307--321, 1999.

\bibitem{grigor2012two}
A.~Grigor'yan and A.~Telcs.
\newblock Two-sided estimates of heat kernels on metric measure spaces.
\newblock {\em Ann. Probab.}, 40(3):1212--1284, 2012.

\bibitem{guillemin2010differential}
V.~Guillemin and A.~Pollack.
\newblock {\em Differential topology}.
\newblock Providence, RI: AMS Chelsea Publishing, reprint of the 1974 original edition, 2010.

\bibitem{gurka1988continuous}
P.~Gurka and B.~Opic.
\newblock Continuous and compact imbeddings of weighted {Sobolev} spaces. {I}.
\newblock {\em Czech. Math. J.}, 38(4):730--744, 1988.

\bibitem{hajlasz1996sobolev}
P.~Haj{\l}asz.
\newblock Sobolev spaces on an arbitrary metric space.
\newblock {\em Potential Anal.}, 5(4):403--415, 1996.

\bibitem{hajlasz2003sobolev}
P.~Haj{\l}asz.
\newblock Sobolev spaces on metric-measure spaces.
\newblock In {\em Heat kernels and analysis on manifolds, graphs, and metric spaces. Lecture notes from a quarter program on heat kernels, random walks, and analysis on manifolds and graphs, April 16--July 13, 2002, Paris, France}, pages 173--218. Providence, RI: American Mathematical Society (AMS), 2003.

\bibitem{SobolevMetPoincare}
P.~Haj{\l}asz and P.~Koskela.
\newblock {\em Sobolev met {Poincar{\'e}}}, volume 688 of {\em Mem. Am. Math. Soc.}
\newblock Providence, RI: American Mathematical Society (AMS), 2000.

\bibitem{han2021asymptotic}
B.-X. Han and A.~Pinamonti.
\newblock On the asymptotic behaviour of the fractional {Sobolev} seminorms in metric measure spaces: {Bourgain}-{Brezis}-{Mironescu}'s theorem revisited.
\newblock Preprint, {arXiv}:2110.05980 [math.{FA}] (2021), 2021.

\bibitem{heinonen2018nonlinear}
J.~Heinonen, T.~Kilpel{\"a}inen, and O.~Martio.
\newblock {\em Nonlinear potential theory of degenerate elliptic equations. {With} a new preface, corrigenda, and epilogue consisting of other new material}.
\newblock Dover Books Math. Mineola, NY: Dover Publications, reprint of the 2006 edition edition, 2018.

\bibitem{kalamajska1999compactness}
A.~Ka{\l}amajska.
\newblock On compactness of embedding for {Sobolev} spaces defined on metric spaces.
\newblock {\em Ann. Acad. Sci. Fenn., Math.}, 24(1), 1999.

\bibitem{kato2013perturbation}
T.~Kato.
\newblock {\em {Perturbation theory for linear operators}}, volume 132.
\newblock Springer Science \& Business Media, 2013.

\bibitem{kigami2004local}
J.~Kigami.
\newblock Local {Nash} inequality and inhomogeneity of heat kernels.
\newblock {\em Proc. Lond. Math. Soc. (3)}, 89(2):525--544, 2004.

\bibitem{kilpelainen1994weighted}
T.~Kilpel{\"a}inen.
\newblock Weighted {Sobolev} spaces and capacity.
\newblock {\em Ann. Acad. Sci. Fenn., Ser. A I, Math.}, 19(1):95--113, 1994.

\bibitem{kinnunen1996sobolev}
J.~Kinnunen and O.~Martio.
\newblock The {Sobolev} capacity on metric spaces.
\newblock {\em Ann. Acad. Sci. Fenn., Math.}, 21(2):367--382, 1996.

\bibitem{Yasuo}
Y.~Komori-Furuya.
\newblock A note on {Muckenhoupt} type weight classes on nondoubling measure spaces.
\newblock {\em Georgian Math. J.}, 18(1):131--135, 2011.

\bibitem{kreuml2019fractional}
A.~Kreuml and O.~Mordhorst.
\newblock Fractional {Sobolev} norms and {BV} functions on manifolds.
\newblock {\em Nonlinear Anal., Theory Methods Appl., Ser. A, Theory Methods}, 187:450--466, 2019.

\bibitem{kumagai2005construction}
T.~Kumagai and K.-T. Sturm.
\newblock Construction of diffusion processes on fractals, {{\(d\)}}-sets, and general metric measure spaces.
\newblock {\em J. Math. Kyoto Univ.}, 45(2):307--327, 2005.

\bibitem{kumura2001nash}
H.~Kumura.
\newblock Nash inequalities for compact manifolds with boundary.
\newblock {\em Kodai Math. J.}, 24(3):352--378, 2001.

\bibitem{kurki2022extension}
E.-K. Kurki and C.~Mudarra.
\newblock On the extension of {Muckenhoupt} weights in metric spaces.
\newblock {\em Nonlinear Anal., Theory Methods Appl., Ser. A, Theory Methods}, 215:20, 2022.
\newblock Id/No 112671.

\bibitem{lahti2023bv}
P.~Lahti, A.~Pinamonti, and X.~Zhou.
\newblock {BV} functions and nonlocal functionals in metric measure spaces.
\newblock {\em J. Geom. Anal.}, 34(10):34, 2024.
\newblock Id/No 318.

\bibitem{lahti2024characterization}
P.~Lahti, A.~Pinamonti, and X.~Zhou.
\newblock A characterization of {BV} and {Sobolev} functions via nonlocal functionals in metric spaces.
\newblock {\em Nonlinear Anal., Theory Methods Appl., Ser. A, Theory Methods}, 241:14, 2024.
\newblock Id/No 113467.

\bibitem{jonatim}
T.~Laux and J.~Lelmi.
\newblock Large data limit of the {MBO} scheme for data clustering: {{\( {{\Gamma}} \)}}-convergence of the thresholding energies.
\newblock {\em Appl. Comput. Harmon. Anal.}, 79:34, 2025.
\newblock Id/No 101800.

\bibitem{Shuwen3D}
L.~Li and S.~Lou.
\newblock Distorted {Brownian} motions on space with varying dimension.
\newblock {\em Electron. J. Probab.}, 27:32, 2022.
\newblock Id/No 72.

\bibitem{Shuwen3D2}
S.~Lou.
\newblock Explicit heat kernels of a model of distorted {Brownian} motion on spaces with varying dimension.
\newblock {\em Ill. J. Math.}, 65(2):287--312, 2021.

\bibitem{Marola_2016}
N.~Marola, M.~j. Miranda, and N.~Shanmugalingam.
\newblock Characterizations of sets of finite perimeter using heat kernels in metric spaces.
\newblock {\em Potential Anal.}, 45(4):609--633, 2016.

\bibitem{miranda2003functions}
M.~Miranda.
\newblock Functions of bounded variation on ``good'' metric spaces.
\newblock {\em J. Math. Pures Appl. (9)}, 82(8):975--1004, 2003.

\bibitem{BMCap}
P.~M{\"o}rters.
\newblock Sample path properties of {Brownian} motion.
\newblock https://www.math-berlin.de/images/stories/lecnotes\_moerters.pdf.
\newblock Accessed: 2025-03-07.

\bibitem{Takumud1d2}
T.~Ooi.
\newblock Heat kernel estimates on spaces with varying dimension.
\newblock {\em T{\^o}hoku Math. J. (2)}, 74(2):165--194, 2022.

\bibitem{ouhabaz2009analysis}
E.~M. Ouhabaz.
\newblock {\em Analysis of heat equations on domains}, volume~31 of {\em Lond. Math. Soc. Monogr. Ser.}
\newblock Princeton, NJ: Princeton University Press, 2005.

\bibitem{paulik2005gluing}
G.~Paulik.
\newblock {\em Gluing spaces and analysis}, volume 372 of {\em Bonn. Math. Schr.}
\newblock Bonn: Univ. Bonn, Mathematisches Institut (Dissertation 2004), 2005.

\bibitem{post2018locality}
O.~Post and R.~R{\"u}ckriemen.
\newblock Locality of the heat kernel on metric measure spaces.
\newblock {\em Complex Anal. Oper. Theory}, 12(3):729--766, 2018.

\bibitem{rosenberg1997laplacian}
S.~Rosenberg.
\newblock {\em The {Laplacian} on a {Riemannian} manifold. {An} introduction to analysis on manifolds}, volume~31 of {\em Lond. Math. Soc. Stud. Texts}.
\newblock Cambridge: Cambridge University Press, 1997.

\bibitem{ruiz2023dirichlet}
P.~A. Ruiz and F.~Baudoin.
\newblock Dirichlet forms on metric measure spaces as {Mosco} limits of {Korevaar}-{Schoen} energies.
\newblock Preprint, {arXiv}:2301.08273 [math.{FA}] (2023), 2023.

\bibitem{stein1993harmonic}
E.~M. Stein.
\newblock {\em Harmonic analysis: {Real}-variable methods, orthogonality, and oscillatory integrals. {With} the assistance of {Timothy} {S}. {Murphy}}, volume~43 of {\em Princeton Math. Ser.}
\newblock Princeton, NJ: Princeton University Press, 1993.

\bibitem{LipChar}
P.~Stollmann.
\newblock A dual characterization of length spaces with application to {Dirichlet} metric spaces.
\newblock {\em Stud. Math.}, 198(3):221--233, 2010.

\bibitem{sturm1995analysis}
K.-T. Sturm.
\newblock Analysis on local {Dirichlet} spaces. {II}: {Upper} {Gaussian} estimates for the fundamental solutions of parabolic equations.
\newblock {\em Osaka J. Math.}, 32(2):275--312, 1995.

\bibitem{sturm1994analysis}
K.-T. Sturm.
\newblock Analysis on local {Dirichlet} spaces. {III}: {The} parabolic {Harnack} inequality.
\newblock {\em J. Math. Pures Appl. (9)}, 75(3):273--297, 1996.

\bibitem{taylor2010partial}
M.~E. Taylor.
\newblock {\em Partial differential equations. {I}: {Basic} theory}, volume 115 of {\em Appl. Math. Sci.}
\newblock New York, NY: Springer, 2nd ed. edition, 2011.

\bibitem{medianfilter}
A.~Ullrich and T.~Laux.
\newblock Median filter method for mean curvature flow using a random {Jacobi} algorithm.
\newblock Preprint, {arXiv}:2410.07776 [math.{AP}] (2024), 2024.

\end{thebibliography}

\end{document}